\documentclass[12pt]{article}
\usepackage{amsmath,amstext,amssymb,amsopn,amsthm,color,undertilde}

\theoremstyle{plain}
\newtheorem{thm}{Theorem}[section]
\newtheorem{lem}[thm]{Lemma}
\newtheorem{prop}[thm]{Proposition}
\newtheorem{cor}[thm]{Corollary}
\theoremstyle{definition}

\newtheorem{exmp}{Example}[section]
\newtheorem*{rem*}{Remark}

\newcommand{\R}{\mathbb{R}}
\newcommand{\N}{\mathbb{N}}

\newcommand{\E}{\mathbb{E}}

\newcommand{\p}{\mathbb{P}}
\renewcommand{\P}{\mathbb{P}}

\renewcommand{\leq}{\leqslant}
\renewcommand{\geq}{\geqslant}

\newcommand{\kernel}{p}

\def\({\left(}
\def\){\right)}
\def\[{\left[}
\def\]{\right]}
\def\<{\langle}
\def\>{\rangle}

\def\d{\delta} 
\def\bs{\backslash}

\title {Laplace Dirichlet heat kernels in convex domains
\footnotetext{2020 MS Classification:
    60J65, 35K08. {\it Key words and phrases}:  Laplacian, heat kernel, Brownian motion, estimates, convex set, exit time.
       }}
\author{  G. Serafin\\
Wroc{\l}aw University of Science and Technology\\
Ul. Wybrze\.ze Wyspia\'nskiego 27\\ Wroc{\l}aw, Poland.}
\date{}

\begin{document}
\maketitle

\begin{abstract} 
We provide general lower and upper bounds for Laplace Dirichlet heat kernel of  convex $\mathcal C^{1,1}$ domains. The obtained estimates precisely describe the exponential behaviour of the kernels, which has been known only in a few special cases so far.  Furthermore, we characterize a class of sets for which the estimates are sharp, i.e. the upper and lower bounds coincide up to a multiplicative constant. In particular, this includes sets of the form  $\{x\in \R^n: x_n>a|(x_1,...,x_{n-1})|^p\}$ where $p\geq2$, $n\geq2 $ and $a>0$. 
\end{abstract}

\section{Introduction}
Heat kernels are basic objects in mathematical analysis, as fundamental solutions to parabolic differential equation (heat equations), as well as in the theory of stochastic processes, as their  transition probability densities. 
They are also, or maybe primarily, important from the point of view of Physics, since they describe evolution of particles, temperature and other phenomena. Despite of a very long and rich history of research on heat kernels in various settings, it turns out that there are still many open question even in the the most classical case, i.e. the one involving  the Laplace operator $\Delta$ (or, equivalently, the Brownian motion) in Euclidean space, which is the subject of this article.

Let $p(t,x,y)=(4\pi t)^{-n/2}e^{-|x-y|^2/4t}$ be the global Laplace heat kernel in $\R^n$, $n\in\N$. For a domain $D\subset \R^n$ we denote by  $p_D(t,x,y)$  the Dirichlet heat kernel of $D$, which is the fundamental solution to the heat equation with the Dirichlet condition at the boundary. From probabilistic point of view, $p_D(t,x,y)$ is the transition probability density of Brownian motion killed when exiting the set $D$. Since applicable explicit formulae for $p_D(t,x,y)$ are known only in a few cases (half-lines, intervals and their products), estimates are strongly desired. They have been intensively  studied for more than  half a century and led to numerous significant results (see, among others,  \cite{CheegerYau:1981, D, GrigorSaloff:2002, MR2807275, MS,  Zhang:2002, Zhang:2003, U,VdB2, VdB3}). In particular, it follows from general theory (\cite{D,DS}) that  for a bounded domain $D$ with boundary smooth enough    the heat kernel $p_D(t,x,y)$ is comparable for large $t$ with $\delta_D(x)\delta_D(y)e^{-\lambda_1t}$, $ x,y\in D$,
where $\lambda_1$ stands for the first eigenvalue of $-\Delta$ in $D$, $\delta_D(x)$ for distance of $x$ to the boundary of $D$. For this reason, we will focus on small times (but will be considering unbounded domains as well).

Most of   the  general heat kernel estimates in the literature  share one common weakness: lower and upper bounds are not comparable and their ratio is usually of the form $e^{|x-y|^2/ct}$ for some constant $c>0$. Such estimates are therefore very imprecise. Let us  recall Zhang's result \cite{Zhang:2002}, which provides the sharpest known bounds for the heat kernel of any set $D$ which is a  bounded $\mathcal C^{1,1}$ domain or a complement of a closure of a bounded $\mathcal C^{1,1}$ set.
Namely,  there are constants $c_1,c_2,c_3,c_4>0$ such that for $x,y\in D$ and $t<T$ it holds 
%\begin{eqnarray*}
%  \label{eq:Davies:est}
%  \left(\frac{\rho_D(x)\rho_D(y)}{t}\wedge 1\right)\frac{c_3}{t^{d/2}}\exp\left({-c_4\frac{|x-y|^2}{t}}\right)  \leq \kkt \leq \left(\frac{\rho_D(x)\rho_D(y)}{t}\wedge 1\right)\frac{c_1}{t^{d/2}}\exp\left({-c_2\frac{|x-y|^2}{t}}\right)
%\end{eqnarray*}
\begin{eqnarray}
  \label{eq:Z}
  c_1\left(\frac{\delta_D(x)\delta_D(y)}{t}\wedge 1\right)\frac{e^{-c_2|x-y|^2/t}}{t^{n/2}}  \leq p_D(t,x,y) \leq c_3\left(\frac{\delta_D(x)\delta_D(y)}{t}\wedge 1\right)\frac{e^{-c_4|x-y|^2/t}}{t^{n/2}}.
\end{eqnarray}
One may observe that incomparability of above-given bounds is caused by two different constants in exponents. Estimates with such a property are known as quantitatively sharp estimates. In fact, there are some results with correct exponents, but they completely fail  in describing the boundary hehaviour (see e.g. \cite{VdB1, VdB2, VdB3}). For instance, the main result of \cite{VdB2} (combined with \cite{MS}), for simplicity restricted to convex domains,  states that
$$p_D(t,x,y)\geq c\(1\wedge \frac{\(\delta_D(x)\wedge \delta_D(y)\)^2}{t}\)\frac{e^{-\lambda t/\(\delta_D(x)\wedge \delta_D(y)\)^2}}{\(1+\frac1t\(\delta_D(x)\wedge \delta_D(y)\)^2\)^{(n+2)/2}}p(t,x,y),$$
for some $c>0$, where $\lambda$ stands for the first eigenvalue of $-\Delta$ in the unit ball. However, the article was focused on asymptotics of the heat kernels with fixed space arguments, where the boundary behaviour plays marginal role. Until recently, precise two-sided estimates for Dirichlet heat kernels have been known only in such basic cases as a half-line and an interval (and their multidimensional extensions) as they are given by simple explicit formulae. 
Even the case of such classical set as a ball turned out to require a more subtle approach and has been solved   in \cite{MS}.  Precisely, for a unit ball $B=B(0,1)$ centered at the origin and   for every  $T>0$ there exists a constant $C=C(n,T)>1$  such that 
\begin{eqnarray}
\label{eq:MS}
\frac{1}{C}\,  h(t,x,y)p(t,x,y)\leq p_B(t,x,y)\leq C\, h(t,x,y)p(t,x,y)
\end{eqnarray}
for every $x,y\in B$ and $t<T$, where 
\begin{eqnarray}
\label{eq:htxy:est}
  h(t,x,y) = \left(1\wedge\frac{\delta_B(x)\delta_B(y)}t\right)+\left(1\wedge\frac{\delta_B(x)|x-y|^2}t\right)\left(1\wedge\frac{\delta_B(y)|x-y|^2}t\right)\/.
\end{eqnarray}
Note that proper description of the exponential behaviour imposed the appearance of a new non-exponential factor $h(t,x,y)$.  Above estimates have been complemented with asymptotics in \cite{Serafin:asymp}, which revealed that the behaviour of $p_B(t,x,y)$ is in fact driven by the expression $\delta\(\frac{x+y}2\)/\sqrt t$. A similar property will be observed in general lower bound \eqref{eq:main2}. We refer the reader to \cite{BGR1, BGR2, BM, GKK, MSZ, NSS1, NSS2} for some other recent articles focused on sharp estimates of heat kernels in other settings.

The goal of this paper is to derive heat kernel estimates with correct exponential behaviour for general $\mathcal C^{1,1}$ convex domains $D$. The first main result is  the  following upper bound (see Theorem \ref{thm:upper})
\begin{align}\nonumber
&p_D(t,x,y)\\\label{eq:main1}
&\leq C p(t,x,y)\[\(1\wedge \frac{\d_D(x)\d_D(y)}{t}\)+\(1\wedge \frac{\d_{H_x}(x)\d_{H_x}\(y\)}{t}\)\(1\wedge \frac{\d_{H_y}(y)\delta_{H_y}\(x\)}{t}\)\],
\end{align}
where $x,y\in D$, $0<t<T$, $C=C(D,T)$ and  $H_x$ is any half-space such that $D\subset H_x$ and $\delta_D(x)=\delta_{H_x}(x)$. Next, we provide lower bounds of the form (see Theorem \ref{cor:mainlower})
\begin{align}\nonumber
&p_D(t,x,y)\\\label{eq:main2}
&\geq C p(t,x,y)\(1\wedge \frac{\d_{D}(x)\(\d_{D}\(\frac{x+y}2\)+\sqrt t\)}{t}\)\(1\wedge \frac{\d_{D}(y)\(\delta_{D}\(\frac{x+y}{2}\)+\sqrt t\)}{t}\),\\\label{eq:main3}
&\approx   p(t,x,y)\[\(1\wedge\frac{\delta_D(x)\delta_D(y)}{t}\)+\(1\wedge \frac{\d_{D}(x)\d_{D}\(\frac{x+y}2\)}{t}\)\(1\wedge \frac{\d_{D}(y)\delta_{D}\(\frac{x+y}{2}\)}{t}\)\],
\end{align}
where $\approx$ means that the ratio of both sides is uniformly bounded and bounded away from zero.
The exponential behaviour is indeed  treated well in each of the above bounds, as expected. Let us now focus on   the non-exponential factors. They  are similar, but not identical. The reason is that they strongly depend on the shape of the boundary of $D$ (see Example \ref{exmpl:1}). Note that in one dimensional case, where any convex set is just an interval with (at most) two-point boundary,  all of the above bounds are equivalent. If there exist elementary sharp estimates in other dimensions,  they are presumably more complex and involving more detailed geometrical features of the set $D$. The main  advantages of  bounds \eqref{eq:main1} -- \eqref {eq:main3} are  therefore  not only the proper exponential behaviour but also a relatively simple form. Furthermore,  they follow sharp estimates for a large class of domains. Indeed, we introduce a class  $\mathcal S_R$ such that the heat kernel of any $D\in \mathcal S_R$ admits two-sided estimates of the form \eqref{eq:main1}.  As an example, we show that  $\mathcal S_R$ contains sets of the form  
$\{x\in \R^n: x_n>a|(x_1,...,x_{n-1})|^p\}$, where $p\geq2$, $n\geq2 $ and $a>0$. Such sets are usually difficult to study, as neither  they are bounded nor their complements are bounded; see \cite{Ban.at.al,Li} for some result concerning the first exit time of such sets and \cite{MR2807275} for quantitatively sharp heat kernel estimates. Then, we   fully characterize a class $\mathcal S_Q\subset \mathcal S_R$ of sets whose heat kernels satisfy two-sided estimates of the forms \eqref{eq:main2} and \eqref{eq:main3}. In particular, $\mathcal S_Q$ contains balls (cf. \eqref{eq:MS}). Finally, let us note that all the obtained bounds immediately imply estimates for the first exit time and place density $q_D^x(t,z)$ of Brownian motion from a domain $D$. Precisely, the well known representation $q_D^x(t,z)=\frac12\frac{\partial}{\partial n_z}p_D(t,x,z)$ 
%(see \cite{H} for $\mathcal C^{3}$ domains, which might be extended to all convex $\mathcal C^{1,1}$ domains by suitable approximation)
  and Dirichlet boundary condition give us
$$q_D^x(t,z)=\frac12\lim_{\substack{D\ni y\rightarrow z\\ (y-z)||n_z}}\frac{p_D(t,x,y)}{\delta_D(y)},$$
where $n_z$ is is the inward normal direction at $z\in \partial D$, which allows us easily transform estimates of $p_D(t,x,y)$ into estimates of $q_D^x(t,z)$.

 The assumption of  smoothness of the boundary of $D$ in the main results is very natural and common in the topic. On the other hand, convexity of the set $D$ is necessary  to obtain exponential behaviour of the same order as in the global heat kernel $p(t,x,y)$.  
Namely, \mbox{S. R. S. Varadhan} showed (Corollary  4.7 in \cite{Var}) that 
$$\lim _{t\rightarrow 0}t\ln\(p_D(t,x,y)\)=\frac14 d_D^2(x,y),$$
where $d_D(x,y)$ is the infimum of lengths of arcs included in $D$ and  connecting $x$ and $y$.  If $D$ was concave, there would be $x,y\in D$ such that $d_D(x,y)>|x-y|$ and consequently $e^{-d^2_D(x,y)/4t}<<e^{-|x-y|^2/4t}$ for $t$ small enough. One could naturally try to obtain estimates with  the term $-d^2_D(x,y)/4t$ in the exponent, but this seem to be a much more challenging  task and rather a material for further research, as there are expected some additional  exponential terms related to  Buslaev conjecture \cite{Bus,IK}. The first attempts of describing heat kernels behavior at points satisfying convexity property, i.e. such that the whole interval $\overline{xy}$ connecting $x$ and $y$ is contained in $D$, are associated with the \emph{property of not feeling the boundary}, introduced by M. Kac in \cite{Kac}, which says that for such  points $x$ and $y$ it holds
$\frac{p_D(t,x,y)}{p(t,x,y)}\stackrel{t\rightarrow0}{\longrightarrow}1$. In \cite{VdB1}, the following  rate of convergence was derived 
\begin{align}\label{eq:vdB}
 p(t,x,y)&\geq p_D(t,x,y)\geq p(t,x,y)\(1-e^{-\rho^2/t}\sum_{k=1}^n\frac{2^k}{(k-1)!}\(\frac{\rho^2}{t}\)^{k-1}\),
\end{align}
where $\rho=\sup_{\substack{w\in \partial D\\z\in \overline{xy}}}|w-z|$ denotes the distance of the interval $\overline{xy}$ to the boundary \mbox{of $D$.} \mbox{A simple} observation is that for $\rho<c \sqrt t$ the bound \eqref{eq:vdB} induces sharp  estimates of $p_D(t,x,y)$. As mentioned before, all the general results in the literature    fail to describe the proper exponential behaviour in the remaining case $\rho<c \sqrt t<C$, and consequently provide bounds with a substantial error. For this reason, the results presented in the paper are first of their kind. One way to explain such enhancement is application of mixture of probabilistic and analytical methods, as purely analytical methods usually lead to estimates with different constants in exponents. Such approach has been already successfully adopted in e.g. \cite{BM, MS, MSZ}. Since also many arguments are geometrical, the methods presented in the paper seem to be adoptable in wider generality.

The paper is organized as follows. In Section 2 we gather  notational details as well as preliminary information about the Brownian motion, including some inequalities related to semi-group property of $p_D(t,x,y)$. Sections 3 and 4 are devoted to general lower and upper bounds, respectively, of the examined heat kernels. Section 5 deals with two-sided estimates and contains some supporting examples.

\section{Preliminaries}

\subsection{Notation}

In this paper we work on the Euclidean space $\R^n$, $n\in\N$, equipped with the standard inner product $x\cdot y$ and corresponding Euclidean metric $|x-y|$, $x,y\in \R^n$.  By $B_k(a,r)$, $k\in\N$, we denote the $k$-dimensional ball of radius $r>0$ and center $a\in\R^k$. If $k=n$, we simply write $B_n(a,r)=B(a,r)$. A set $D\subset \R^n$ is  called a domain,    if it is  open and connected. The $\mathcal{C}^{1,1}$ sets are sets whose boundary is locally a graph of a $\mathcal C^{1,1}$ function $f:\R^{n-1}\rightarrow\R$. It is well known that $\mathcal{C}^{1,1}$ sets satisfy the inner and outer ball condition, which means that for any point $z$ from the boundary $\partial D$ of the set $D$ there are two balls tangent to $D$ at $z$ such that one of them is completely included in $D$, and the other one in $D^c$. Furthermore, if $D$ is bounded, then there exists a radius $r>0$ such  that for any $z\in \partial D$ the condition is satisfied with balls of radius at least $r$.  We will denote the class of sets with such property \mbox{by $\mathcal C^{1,1}_r(\R^n)$}.

For a domain $D$ and $x\in D$ we denote the distance of $x$ to the boundary $\partial D$ of $D$ by $\delta_D(x)$. For a convex domain $D$ and  $z\in \partial D$ from its boundary,  $P_z$ stands for a hyperplane such that $P_z\subset D^c$ and $\{z\}\in P_z$. Note that $P_z$ might be not unique, but if $D$ is a $\mathcal C^{1,1}$ domain, then there is only one such hyperplane and it is tangent to $D$ at $z$. Any hyperplane $P_z$ divides the whole space $\R^n$ into two half-spaces. The one including $D$ will be denoted by $H_z$. In that case we have $\partial H_z=P_z$. For $x\in D$ and $z\in \partial D$ realising the distance of $x$ to the boundary of $D$, i.e. $\delta_D(x)=|x-z|$, we put $P_x=P_z$ and $H_x=H_z$.   Again, there might be more than one such point $z$, but this is irrelevant from our point of view, as all the results  presented in the paper are valid for any choice of $z$. For two half-spaces $H_1$ and $H_2$ we define the angle $\angle(H_1,H_2)$ between them as the angle inside $H_1\cap H_2$ between  lines $l_1\subset \partial H_1$ and $l_2\subset \partial H_2$ that are perpendicular to $\partial H_1\cap \partial H_2$. When $\angle(H_1,H_2)\leq \pi$ then the angle is equal to $\pi-\angle(v_1,v_2)$, where $v_1,v_2$ are normal vectors of $\partial H_1$ and $\partial H_2$, respectively, directed inside the set $H_1\cap H_2$. 

For $x=(x_1,...,x_n)\in \R^n$, $n\geq 2$, we denote 
$$\widetilde x=(x_2,...,x_n)\in \R^{n-1},\ \ \ \ \text{and}\ \ \ \  \ \utilde x =(x_1,...,x_{n-1})\in \R^{n-1}.$$
If $x\in \R$, then $\widetilde x$ and $\utilde x$ will be treated as $0$ in calculations. Analogously we define  $\widetilde{\utilde x}$.

To compare two positive functions $f,g$ we use notation $f\approx g$, which means that there exist  constants $c_1,c_2>1$, possibly depending on $n$, such that $c_1<f/g<c_2$ for a given range of arguments. If the constants depend on other parameters $p_1,...,p_k$, $k\in \N$, we write $f\stackrel{p_1,...,p_k}{\approx} g$.

\subsection{Brownian motion}
Let us consider $n$-dimensional Brownian motion $W = (W(t))_{t\geq0}= (W_1(t),...,W_n(t))_{t\geq0}$ starting
from $x \in \R^n$.  The global heat kernel $p(t,x,y)$ represents its  transition probability density. By $\P^x$ and $\E^x$ we denote the corresponding probability law and the expected value, respectively.

 For a domain  $D\subset \R^n$ we define   the first exit time $\tau_D$ of $W$ from $D$  by
$$\tau_D:=\inf\{t>0:W(t)\notin D\}.$$
Then by $W^D=(W^D_t)_{t\geq 0}$ we denote the Brownian motion killed  upon leaving a set $D$, which is a process equal to $W_t$ before time $\tau_D$ and at that time it  is moved to an additional state called $cemetery$.  For sufficiently regular domains $D$ (e.g. Lipschitz domains) the transition density function of $W^D$ is given  by the Dirichlet heat kernel $p_D(t,x,y)$. The relation between $p_D(t,x,y)$ and $p(t,x,y)$ is described  by the Hunt formula
\begin{eqnarray}
\label{eq:Hunt}
   p_D(t,x,y) = p(t,x,y)-\int_0^t \int_{\partial D}p(t-s,z,y)q^x_D(s,z)dsd\sigma(z)\/,
\end{eqnarray}
where $\sigma(z)$ is the surface measure on $\partial D$ and  $q_x(t,z)$ denotes the density function of the joint distribution $(\tau_D,W_{\tau_D})$ for the process $W$ starting from $x\in D$. Note that the function $p_D(t,x,y)$ is symmetric in space arguments and satisfies the Chapman-Kolmogorov identity, known also as the semi-group property, (see Theorem 2.4 in \cite{CZ})
\begin{align}\label{eq:CK}
p_D(t,x,y)=\int_{D}\kernel_D(\alpha t,x,z)\kernel_D((1-\alpha)t,z,y)\,dz,\ \ \ \ \  x,y\in D,\ t>0,\  \alpha\in(0,1).
\end{align}

For any half-space $H\subset \R^n$ the Dirichlet heat kernel $p_H(t,x,y)$ takes especially simple form. Precisely, reflection principle gives us
\begin{align*}
p_H(t,x,y)&=p(t,x,y)-p(t,x,\bar y)\\
&=p(t,x,y)\(1-e^{-\delta_H(x)\delta_H(y)}\), 
\end{align*}
where $\bar y$ is a symmetric reflection of $y$ with respect to the boundary $\partial H$ of $H$. This immediately implies 
\begin{eqnarray}\label{eq:estH}p_H(t,x,y) \approx\(1\wedge\frac{\delta_H(x)\delta_H(y)}{t}\)p(t,x,y).\end{eqnarray}

Another important property of Dirichlet heat kernels, which follows e.g. from the Hunt formula \eqref{eq:Hunt},  is their monotonicity with respect to inclusion of domains. Namely, if $D_1\subset D_2$ then
\begin{align}\label{eq:D1<D2}
p_{D_1}(t,x,y)\leq p_{D_2}(t,x,y),\ \ \ \ x,y\in D_1.
\end{align}
To see this, let us observe that  $\tau_{D_1}\leq \tau_{D_2}$ and consequently $\{\tau_{D_1}>t\}\subset \{\tau_{D_2}>t\}$. Hence,  for any borel $A\subset D_1$ we have
  \begin{align*}
  \int_A& p_{D_1} (t,x,y)dy=\p^x\left[W_t\in A, \tau_{D_1}>t\right]\leq \p^x\left[W_t\in A, \tau_{D_2}>t\right]=\int_A \kernel_{D_2}(t,x,y)dy.
  \end{align*} 

\subsection{Chapman-Kolmogorov-like inequalities}

Below we collect some inequalities related to the Chapman-Kolmogorov identity (semi-group property) for the global heat kernel $p(t,x,y)$, that help us with dealing with analogous identities in case of  the Dirichlet heat kernels and provide some  intuitions about typical trajectories of Brownian motion. The first proposition may be interpreted by saying that Brownian motion going from $x$ to $y$ in time $t$ is mostly at time $\alpha t$, $\alpha\in(0,1)$, passing through a neighbourhood of the point $(1-\alpha)x+\alpha y$ of a size comparable to $\sqrt{\alpha(1-\alpha) t}$. In fact, we can move away from $(1-\alpha)x+\alpha y$ by a distance comparable to $\sqrt{\alpha(1-\alpha) t}$. 
\begin{prop}
  \label{cor:CKlow}
  For every $r,d\geq0$ and $\alpha\in(0,1)$ we have
	\begin{align*}
	   &\int_{B(a,r)}\kernel(\alpha t,x,z)\kernel((1-\alpha)t,z,y)\,dz\geq \frac{e^{-\frac{d^2}{2\alpha(1-\alpha)t}-1}}{2^n\Gamma\(\frac {n+2}2\)} \(1\wedge \frac {r^2}{\alpha(1-\alpha)t}\)^{n/2}p(t,x,y),
	\end{align*}
where $x,y\in \R^n$, $t>0$ and $a\in \R^n$ such that $|a-((1-\alpha) x+\alpha y)|= d$.
\end{prop}
\begin{proof}
Without loss of generality we assume $x=(-\alpha|x-y|,0,...,0)$,  $y=((1-\alpha)|x-y|,0,...,0)$. This follows \begin{eqnarray*}
  \frac1\alpha|x-z|^2+\frac1{1-\alpha}|y-z|^2 &=& |x-y|^2 + \frac{1}{\alpha(1-\alpha)}|z|^2,\ \ \ \ \ \ z\in\R^n,
\end{eqnarray*}
and consequently
\begin{align}\nonumber
  \kernel(\alpha t,x,z)\kernel((1-\alpha)t,z,y)&= \frac{\exp\left(-\frac{|x-y|^2}{4t}\right)\exp\left(-\frac{|z|^2}{4\alpha(1-\alpha)t}\right)}{(4\pi t)^{n/2}(4\alpha(1-\alpha)\pi t)^{n/2}}\\\label{eq:pp}
  &=p(t,x,y) \frac{\exp\left(-\frac{|z|^2}{4\alpha(1-\alpha)t}\right)}{(4\alpha(1-\alpha)\pi t)^{n/2}}\/.
\end{align}
Furthermore, the special form of $x$ and $y$ gives us $(1-\alpha) x+\alpha y=0$,  which implies  
 $$|z|^2\leq\(|z-a|+|a|\)^2\leq 2\(|z-a|^2+|a|^2\)=2\(|z-a|^2+d^2\).$$
Thus, we obtain 
\begin{align*}
   &\int_{B(a,r)}\kernel(\alpha t,x,z)\kernel((1-\alpha)t,z,y)dz\\
& \geq \frac{p(t,x,y)}{(4\alpha(1-\alpha)\pi t)^{n/2}}\int_{B(a,r)}\exp\left(-\frac{|z-a|^2+d^2}{2\alpha(1-\alpha)t}\right)dz\\
& =\frac{p(t,x,y)}{(2\pi)^{n/2}}e^{-\frac{d^2}{2\alpha(1-\alpha)t}}\int_{B(0,r/\sqrt{2\alpha(1-\alpha)t})}\exp\left(-{|z|^2}\right)dz\\
& \geq \frac{p(t,x,y)}{(2\pi)^{n/2}}e^{-\frac{d^2}{2\alpha(1-\alpha)t}}e^{-1^2}\left|B\(0,1\wedge\(r/\sqrt{2\alpha(1-\alpha)t}\)\)\right|\\
&= \frac{p(t,x,y)}{2^{n/2}\Gamma\(\frac {n+2}2\)}e^{-\frac{d^2}{2\alpha(1-\alpha)t}-1}\(1\wedge \frac {r}{\sqrt{2\alpha(1-\alpha)t}}\)^{n}\\
&\geq \frac{p(t,x,y)}{2^n\Gamma\(\frac {n+2}2\)} e^{-\frac{d^2}{2\alpha(1-\alpha)t}-1}\(1\wedge \frac {r^2}{\alpha(1-\alpha)t}\)^{n/2},
\end{align*}
as required
\end{proof}
The next proposition deals with an integral over the intersection of two half-spaces, where integrands are somewhat related to the Dirichlet heat kernels of the half-spaces (cf. \eqref{eq:estH}).
\begin{prop}\label{prop:CKHH}
Let $\alpha, \beta\geq0$  and  $H_1,H_2\subset \R^n$ be two half-spaces. For $D=H_1\cap H_2$  it holds   
\begin{align*}
&\int_{D} p(t/2,x,z)p(t/2,z,y)\(\delta_{H_1}(z)\)^\alpha \(\delta_{H_2}(z)\)^\beta dz\\
&\stackrel{\alpha,\beta}{\lesssim} \,p(t,x,y)\(\sqrt t+{\delta_{H_1}\(\frac{x+y}2\)}\)^{\alpha}\(\sqrt t+{\delta_{H_2}\(\frac{x+y}2\)}\)^{\beta}.
\end{align*}
\end{prop}

\begin{proof}
Similarly as in the proof of Proposition \ref{cor:CKlow}, we assume $x=(-\frac12|x-y|,0,...,0)$,  $y=(\frac12|x-y|,0,...,0)$. Then we have $B(( x+ y)/2,r)=B(0,r)$ and  
$$\delta_{H_1}(z)\leq |z|+\delta_{H_1}\(\frac{x+y}2\),\ \ \ \ \ \ \ \delta_{H_2}(z)\leq |z|+\delta_{H_2}\(\frac{x+y}2\),$$
 as well as
\begin{eqnarray*}
 p (t/2,x,z)p(t/2,z,y)=p(t,x,y) \frac{e^{-|z|^2/t}}{(\pi t)^{n/2}}\/.
\end{eqnarray*}
This follows
\begin{align*}
&\int_{D} p(t/2,x,z)p(t/2,z,y)\(\delta_{H_1}(z)\)^\alpha \(\delta_{H_2}(z)\)^\beta dz\\
&\leq p(t,x,y)\int_{\R^n}\frac{e^{-|z|^2/t}}{(\pi t)^{n/2}}\(|z|+\delta_{H_1}\(\frac{x+y}2\)\)^{\alpha}\(|z|+\delta_{H_2}\(\frac{x+y}2\)\)^{\beta} dz\\
&= p(t,x,y)\int_{\R^n}\frac{e^{-|z|^2}}{\pi ^{n/2}}\(|z|\sqrt t+\delta_{H_1}\(\frac{x+y}2\)\)^{\alpha}\(|z|\sqrt t+\delta_{H_2}\(\frac{x+y}2\)\)^{\beta} dz\\
        &\stackrel{\alpha,\beta}{\lesssim} p(t,x,y)\(\sqrt t+{\delta_{H_1}\(\frac{x+y}2\)}\)^{\alpha}\(\sqrt t+{\delta_{H_2}\(\frac{x+y}2\)}\)^{\beta},
\end{align*}
 where the last estimate may be justified by the inequalities 
$$\frac12 \(a^{\gamma}+b^\gamma\)\leq (a+b)^{\gamma}\leq 2^\gamma \(a^{\gamma}+b^\gamma\),$$
for $a,b,\gamma >0$.
\end{proof}

\section{Upper bounds}
Our general idea of finding upper bounds is to circumscribe sets of the form $H_1\cap H_2$, where $H_1$ and $H_2$ are two half-spaces, on the domain $D$ and to use the monotonicity property \eqref{eq:D1<D2}. Due to independence of coordinates of Brownian motion, the heat kernel of $H_1\cap H_2$ is just a product of a heat kernel of a two-dimensional cone and the $n-2$-dimensional global heat kernel. However, in the literature there are neither estimates of heat kernels in cones that describe properly the exponential behaviour nor ones that are  uniform with respect to the angle between the half-spaces (see \cite{BS} for some formulae and properties of Brownian motion in cones). For this reason we provide some upper bounds that are sufficient for applications to estimation of heat kernels in  $C^{1,1}$ domains.
We deal separately with cases when the angle between the half-spaces is 
obtuse or acute. 

\begin{thm}\label{thm:upper1} Let $H_1$ and $H_2$ be two half-spaces such that $\angle(H_1,H_2)\geq \pi/2$  and  denote $D=H_1\cap H_2$. Then there is an absolute constant  $C>0$  such that for $x,y\in D$ and $t>0$ it holds 
\begin{align}\nonumber
p_D(t,x,y)\leq C&\, p(t,x,y)\times\\
\label{eq:aux1}
&\[\(1\wedge \frac{\d_D(x)\d_D(y)}{t}\)+\(1\wedge \frac{\d_{H_1}(x)\d_{H_1}(y)}{t}\)\(1\wedge \frac{\d_{H_2}(x)\delta_{H_2}(y)}{t}\)\].
\end{align}
\end{thm}
\begin{proof}
%$$p_D(t,x,y)\leq p_{H_x\cap H_y}(t,x,y).$$

Without loss of the generality we assume  $\d_{H_1}(x)\leq\d_{H_2}(x)$. Since $\delta_D(x)=\d_{H_1}(x)\wedge\d_{H_2}(x)$, we have $\d_D(x)=\d_{H_1}(x)$. 

If   $\d_{H_1}(y)\leq2\d_{H_2}(y)$, then $\d_D(y)\geq\frac12\d_{H_1}(y)$, and consequently
$$\d_{H_1}(x)\d_{H_1}(y)\leq2\d_D(x)\d_D(y).$$
Hence, the inequality 
$p_D(t,x,y)\leq p_{H_1}(t,x,y)$ together with the estimate \eqref{eq:estH} finish the proof in this case. 

Consider now  $\d_{H_1}(y)>2\d_{H_2}(y)$ and let $\bar x$, $\bar y$ be reflections of $x$, $y$ with respect to hyperplanes $\partial H_1$, $\partial H_2$, respectively. Since $|y-\bar y|=2\d_{H_2}(y)$, we have $\bar y\in H_1$. We may therefore repeat the argumentation of the formula (3.2) in \cite{MS} and get
\begin{align*}
p_D(t,x,y)&\leq p_{H_1\cap H_2}(t,x,y)\\
&\leq p(t,x,y)-p(t,\bar x, y) -p(t,x,\bar y)+p(t,\bar x, \bar y).
\end{align*}
Note that for $\angle(H_1,H_2)\leq \frac12\pi$ the last inequality is expected to be opposite. Next,  the right-hand side may be   rewritten as
\begin{align*}
 &p(t,x,y)\(1-\frac{p(t,\bar x, y)}{p(t, x, y)}\)\(1-\frac{p(t,x,\bar y)}{p(t,x,y)}\) \\
&\ \ \ \ \ \ \ \ \ \ \ \ \ \ \ \ \ +p(t,\bar x,\bar y)-\frac{p(t,\bar x,y)p(t,x,\bar y)}{p(t, x, y)}\\
&= \frac{p_{H_1}(t,x,y)p_{H_2}(t,x,y)}{p(t,x,y)}+\(p(t,\bar x,\bar y)-\frac{p(t,\bar x,y)p(t,x,\bar y)}{p(t, x, y)}\).
\end{align*}
By \eqref{eq:estH}, we clearly have 
$$\frac{p_{H_1}(t,x,y)p_{H_2}(t,x,y)}{p(t,x,y)}\approx p(t,x,y)\(1\wedge \frac{\d_{H_1}(x)\d_{H_1}(y)}{t}\)\(1\wedge \frac{\d_{H_2}(x)\delta_{H_2}(y)}{t}\).$$
In order to deal with the other component, let us observe
\begin{align*}
|\bar x-y|^2&=|(x-y)+(\bar x-x)|^2\\
&=|x-y|^2+\(\bar x-x\)\cdot\big((\bar x-x)+2(x-y) \big),\\
|\bar x-\bar y|^2&=|(x-y)+(\bar x-x)-(\bar y-y)|^2\\
&=|x-y|^2-2\(\bar x -x\)\cdot \(\bar y-y\)\\
&\ \ \ +\(\bar x -x\)\cdot\big( (\bar x -x)+2(x-y)\big)+\(\bar y-y\)\cdot\big( (\bar y -y)+2(y-x)\big),
\end{align*}
which follows
\begin{align*}
p&(t,\bar x,\bar y)-\frac{p(t,\bar x,y)p(t,x,\bar y)}{p(t, x, y)}=p(t,\bar x,\bar y)\(1-e^{-(\bar x-x)\cdot(\bar y-y)/2t}\).
\end{align*}
%Furthermore, the assumption $\angle(P_x,P_y)\geq \pi/2$ gives us 
%$$|\angle(\bar x-x,\bar y-y)|=\pi-\angle(P_x,P_y)\leq\pi/4.$$ 
%As a consequence, we have
Since, by Cauchy-Schwarz inequality,
$$(\bar x-x)\cdot(\bar y-y)\leq |\bar x-x||\bar y-y|=4\d_{H_1}(x)\d_{H_2}(y)=4\d_D(x)\d_D(y),$$
we finally obtain
\begin{align*}
p(t,\bar x,\bar y)\(1-e^{-(\bar x-x)(\bar y-y)/t}\)&\leq p(t,x,y)\(1-e^{-2\d_D(x)\d_D(y)/t}\)\\
& \lesssim p(t,x,y)\(1\wedge\frac{\d_D(x)\d_D(y)}{t}\),
\end{align*}
as required.  The proof is complete.
\end{proof}
In all the subsequent results $H_x, H_y$, $x,y\in D$ stand for  two half-spaces  such that $D\subset H_x\cap H_y$ and $\delta_{H_x}(x)=\delta_D(x)$, $\delta_{H_y}(y)=\delta_D(y)$.

\begin{cor}\label{cor:upper1}  Let $D\subset \R^n$ be any convex domain. There is an absolute constant  $C>0$  such that 
\begin{align}\nonumber
p_D(t,x,y)&\leq C p(t,x,y)\[\(1\wedge \frac{\d(x)\d(y)}{t}\)+\(1\wedge \frac{\d_{H_x}(x)\d_{H_x}(y)}{t}\)\(1\wedge \frac{\d_{H_y}(x)\delta_{H_y}(y)}{t}\)\]
\end{align}
holds whenever  $\angle(H_x,H_y)\geq \frac12\pi$.
\end{cor}

Let  us pass to the  latter main result of this section.

\begin{thm}\label{thm:upper2}Let $D\subset \R^n$ be a $\mathcal C^{1,1}_r$, $r>0$, domain. For $0<t<T$ we have 
\begin{align}\label{eq:smallangle1}
p_D(t,x,y)&\lesssim \(1+\frac T{r^2}\) p(t,x,y)\(1\wedge \frac{\d_{H_x}(x)\d_{H_x}(y)}{t}\)\(1\wedge \frac{\d_{H_y}(x)\delta_{H_y}(y)}{t}\).
\end{align}
\end{thm}
\begin{proof}
 Let $x',y'\in \partial D$ be tangent points of $D$ to $H_x$ and $H_y$, respectively, i.e. such that $\delta_D(x)=\delta_{H_x}(x)=|x-x'|$ and $\delta_D(y)=\delta_{H_y}(y)=|y-y'|$. Since $D$ is $\mathcal C_r^{1,1}$, there are balls of radius $r$ inside $D$ which are tangent at $x'$ and $y'$ to $ D$ (and consequently to $H_x$ and $H_y$, respectively).  If we combine this with the assumption $\angle(H_x,H_y)< \pi/2$, simple geometry shows that  
\begin{align}\label{eq:dd>r/2}
\d_{H_x}(y), \d_{H_y}(x)\geq r.
\end{align}
Due to symmetry of $p_D(t,x,y)$, we may  assume $\delta_D(x)\leq \delta_D(y)$. 
If $\delta_D(y)=\delta_{H_y}(y)>r$, then it holds $\(1\wedge \frac{\delta_{H_y}(x)\d_{H_y}(y)}{t}\)\geq \(1\wedge\frac{r^2}T\)$, and consequently, by \eqref{eq:estH} and \eqref{eq:D1<D2},
\begin{align*}
p_D(t,x,y)&\leq p_{H_x}(t,x,y)\lesssim p(t,x,y)\(1\wedge \frac{\delta_{H_x}(x)\d_{H_x}(y)}{t}\)\\
&\leq \(1+ \frac T{r^2}\) p(t,x,y)\(1\wedge \frac{\delta_{H_x}(x)\d_{H_x}(y)}{t}\)\(1\wedge \frac{\delta_{H_y}(x)\d_{H_y}(y)}{t}\),
\end{align*}
as required. 

Consider now $\delta_D(y)\leq r$. By the assumption $\delta_D(x)\leq \delta_D(y)$ we also have $\delta_D(x)\leq r$. Combining this with \eqref{eq:dd>r/2} we get $\delta_{H_x}(x)\leq \delta_{H_x}(y)$ and  $\delta_{H_y}(y)\leq \delta_{H_y}(x)$. Thus,  equality $\delta_{H_x}\(\frac{x+y}2\)=\frac12\(\delta_{H_x}(x)+\delta_{H_x}(y)\)$ gives us
\begin{align}\label{eq:R<<}
\frac r2\leq \delta_{H_x}\(\frac{x+y}2\)\leq \d_{H_x}(y),\ \ \ \ \ \frac r2\leq \delta_{H_y}\(\frac{x+y}2\)\leq \d_{H_y}(x).
\end{align}
Next, using    Chapman-Kolmogorov identity and estimating $p_D(t/2,x,z)\leq p_{H_x}(t/2,x,z)$, $p_D(t/2,z,y)\leq p_{H_y}(t/2,z,y)$  we obtain for every $i,j\in\{0,1\}$
\begin{align*}
p_D(t,x,y)&\lesssim\int_D\(1\wedge \frac{\d_{H_x}(x)\d_{H_x}\(z\)}{t}\)\(1\wedge \frac{\d_{H_y}(y)\delta_{H_y}\(z\)}{t}\)p(t/2,x,z)p(t/2,x,y)dz\lesssim I_{i,j},
\end{align*}
where
\begin{align*}
 I_{i,j}:=\int_{H_x\cap H_y}\(\frac{\d_{H_x}(x)\d_{H_x}\(z\)}{t}\)^i\(\frac{\d_{H_y}(y)\delta_{H_y}\(z\)}{t}\)^jp(t/2,x,z)p(t/2,x,y)dz.
\end{align*}
Applying Proposition \ref{prop:CKHH} and the formula \eqref{eq:R<<} we get 
\begin{align*}
I_{i,j}&\lesssim p(t,x,y)\(\frac{\delta_{H_x}(x)\(\sqrt T+\delta_{H_x}\(\frac{x+y}2\)\)}{ t}\)^i\(\frac{\delta_{H_y}(y)\(\sqrt T+\delta_{H_x}\(\frac{x+y}2\)\)}{ t}\)^j\\
&\lesssim \(1+\frac T{r^2}\)p(t,x,y)\(\frac{\delta_{H_x}(x)\delta_{H_x}\(y\)}{ t}\)^i\(\frac{\delta_{H_y}(y)\delta_{H_y}\(x\)}{ t}\)^j,
\end{align*}
which follows
\begin{align*}
p_D(t,x,y)&\lesssim \(1+\frac T{r^2}\)p(t,x,y)\min_{i,j\in\{0,1\}}\left\{\(\frac{\delta_{H_x}(x)\delta_{H_x}\(y\)}{ t}\)^i\(\frac{\delta_{H_y}(y)\delta_{H_y}\(x\)}{ t}\)^j\right\}\\
&=\(1+\frac T{r^2}\)p(t,x,y)\(1\wedge\frac{\delta_{H_x}(x)\delta_{H_x}\(y\)}{ t}\)\(1\wedge\frac{\delta_{H_y}(y)\delta_{H_y}\(x\)}{ t}\).
\end{align*}
This ends the proof.

\end{proof}

Theorem \ref{thm:upper} follows now directly from Corollary \ref{cor:upper1} and Theorem \ref{thm:upper2}.

\begin{thm}\label{thm:upper}
Let $D\subset \R^n$ be a $\mathcal C^{1,1}_r$, $r>0$, domain. There is a constant  $C=C(T,n,r)>0$  such that 
\begin{align}\nonumber
p_D(t,x,y)&\leq C p(t,x,y)\[\(1\wedge \frac{\d(x)\d(y)}{t}\)+\(1\wedge \frac{\d_{H_x}(x)\d_{H_x}\(y\)}{t}\)\(1\wedge \frac{\d_{H_y}(y)\delta_{H_y}\(x\)}{t}\)\],
\end{align}
where $x,y\in D$, $t<T$.
\end{thm}

Using the inequality $\delta_{H_x}(y)\leq \(\delta_{H_x}\(x\)+\delta_{H_x}(y)\)\leq  2\delta_{H_x}\(\frac{x+y}2\)$ we obtain another bound, which will be used in the sequel.

\begin{cor}\label{cor:mainupper} Let $D\subset \R^n$ be a $\mathcal C^{1,1}_r$, $r>0$, domain. There is a constant  $C=C(T,n,r)>0$  such that 
\begin{align*}
&p_D(t,x,y)\\
&\leq C p(t,x,y)\[\(1\wedge \frac{\d(x)\d(y)}{t}\)+\(1\wedge \frac{\d_{H_x}(x)\d_{H_x}\(\frac{x+y}2\)}{t}\)\(1\wedge \frac{\d_{H_y}(y)\delta_{H_y}\(\frac{x+y}2\)}{t}\)\],
\end{align*}
where $x,y\in D$, $t<T$.
\end{cor}

\section{Lower bounds}
We start this section with a lower bound for heat kernels of a very specific set, which is so called half-capsule. Namely, for $L\geq R>0$ we define
\begin{equation}\label{eq:S}
J_{R,L}=B_n(0,R)\cup \((0,L)\times B_{n-1}(0,R)\).
\end{equation}
Roughly speaking, $J_{R,L}$ is a cylinder of  radius $R$ and  height $L$ with a hemisphere of radius $R$ attached to one on the bases of the cylinder. 
\begin{lem} \label{lem:S1L}
Let  $L>3\sqrt t$ and $x=(L-\sqrt t,0,0,...,0)$. There is a constant $C$ depending only on $n$  such that for $0<s\leq t$ and  $y\in J_{\sqrt t, L}$ such that $y_1\leq 0$ we have
\begin{align}\label{eq:S1L}
p_{J_{\sqrt t,L}}(s,x,y)\geq C\(1\wedge\frac{\delta_{j_{\sqrt t,L}}(y)\sqrt t}{s}\)p(s,x,y).
\end{align}
\end{lem}
\begin{proof}
The scaling property ofBrownian motion gives us
$$p_{J_{\sqrt t,L}}(s,x,y)=p_{J_{1, L/\sqrt t}}\(\frac st,\frac{x}{\sqrt t},\frac y{\sqrt t}\),$$
which, together with the equality $\delta_{J_{\sqrt t,L}}(y)=\sqrt t\delta_{J_{1,L/\sqrt t}}\(\frac y{\sqrt t}\)$, allows us  to consider only the case $t=1$.  Without loss of generality we also assume $y=(y_1,y_2,0,...,0)$ with $y_1\leq0$ and $y_2\in [0,1)$, which is justified by the  rotational invariance of the set $J_{1,L}$ with respect to the  $Ox_1$ axis. 

Recall the notation
$\R^n\ni x\longrightarrow \widetilde x=(x_2,x_3,...,x_n)\in\R^{n-1}$
and let
$$\mathbb T=\left\{x\in\R^n: \(|\widetilde y|-\tfrac34\)^2+y_1^2<\(\tfrac14\)^2\right\}\subset J_{1,L} $$
be a torus tangent to $J_{1,L}$ at $\{0\}\times \partial B_{n-1}(0,1)$, and let  $\widetilde{ \mathbb T}:=\mathbb T\cup \((-\frac14,\frac14)\times B_{n-1}(0,\frac34)\)$ be the set of  all convex combinations of  points from $\mathbb T$. Alternatively, we may define $\widetilde{ \mathbb T}$ as
$$\widetilde{ \mathbb T}=\bigcup_{z\in\{0\}\times B_{n-1}(0,\frac34)}B_n\(z,\tfrac14\),$$
which may better depict properties of this set.

The main step of the proof is to show that the assertion holds for $y\in\widetilde {\mathbb T}\cap\{y_1<0\}$. For $y\in \widetilde {\mathbb T}\bs  {\mathbb T}\subset (-\frac14,\frac14)\times B_{n-1}(0,\frac34)$ it holds $\delta_{J_{1,L}}(y)>\sqrt{5/8}$. Combining this with $\delta_{J_{1,L}}(x)=1$ and convexity of $J_{1,L}$, one may conclude  \eqref{eq:S1L} from \eqref{eq:vdB}. 

For $y\in \mathbb T\cap\{y_1<0\}$ the argument is  more complicated.
Let $\tau_1$ be the first time of hitting the hyperplane $\{z\in\R^n:z_1=0\}$ by the Brownian motion $W$  starting from $x$, i.e.
$$\tau_1:=\inf\{u>0:W_1(u)=0\}.$$
Then, for any Borel set $A\subset\{z\in J_{1,L}:z_1\leq0\}$ the inclusion $\{W(s)\in A\}\subset \{s\geq\tau_{1}\}$, $s>0$, and Strong Markov property give us
\begin{align}\notag
\int_Ap_{J_{1,L}}(s,x,z)dz
&=\p^x\(W(s)\in A, s<\tau_{J_{1,L}}\)\\\notag
&=\p^x\(W(s)\in A, s<\tau_{J_{1,L}};s\geq\tau_{1}\)\\\notag
&=\E^x\[s<\tau_{J_{1,L}};s\geq\tau_{1};\E^{W(u)}\[s-u<\tau_{J_{1,L}};W(s-u)\in A\]_{u=\tau_{1}}\]\\\label{eq:longMarkov}
&=\E^x\[s<\tau_{J_{1,L}};s\geq\tau_{1};\int_Ap_{J_{1,L}}(s-\tau_{1},W(\tau_{1}),z)dz\],
\end{align}
which implies
\begin{align}
\label{eq:shortMarkov}
&p_{J_{1,L}}(s,x,y)=\E^x\[s<\tau_{J_{1,L}};s\geq\tau_{1};p_{J_{1,L}}(s-\tau_{1},W(\tau_{1}),z)\],
\end{align}
where $y\in \{z\in J_{1,L}:z_1\leq0\}$. Let us now  introduce another few sets. First, we denote 
\begin{align*}
B_0&= B((0,\tfrac14,0,...,0), \tfrac34),\\
I_0&= B_{n-1}((\tfrac14,0,...,0), \tfrac34).
\end{align*}
The relation between $B_0$ and $I_0$ is that $B_0\cap\{x_1=0\}=\{0\}\times I_0$. Furthermore, we define
\begin{align*}
R^{(1)}_L&:=(-\infty,L)\times B_{n-1}(0,1),\\
R^{(2)}_L&:=(-\infty,L)\times I_0,\\
H&:=\{x\in \R^n:x_2>1\}.
\end{align*}
The set $R^{(1)}_L$ is an extension of $J_{1,L}$ into a half-infinite cylinder, $R^{(2)}_L$ is another  half-infinite cylinder which is contained in $R^{(1)}_L$  and $H$ is a half-space tangent to both of the cylinders at $(-\infty,L)\times\{(1,0,...,0)\}\subset \R^n$.
The crucial properties of these sets are 
\begin{align}
\label{eq:dRRH}
\delta_{R^{(1)}_L}(x)\approx \delta_{R^{(2)}_L}(x)\approx\delta_H(x)\approx 1,
\end{align}
and   
\begin{eqnarray}
\label{eq:dBdR}
\delta_{B(0,1)}(z)&\approx& \delta_{R^{(1)}_L}(z)\approx \delta_H(z),\ \ \ \ \ \ \ \text{ for } z\in B_0.
\end{eqnarray}
Bounds in \eqref{eq:dRRH} are clear while bounds in \eqref{eq:dBdR} follow from the below-given calculations:
\begin{align}\notag
\delta_H(z)&\geq\d_{R^{(1)}_L}(z)\geq \delta_{B(0,1)}(y)=1-|z|\geq \tfrac12\(1-|z|^2\)\\\notag
&=\tfrac12\[\tfrac12\(1-z_2\)-\(z_1^2+\(z_2-\tfrac14\)^2+...+z_n^2-\(\tfrac34\)^2\)\]\\\notag
&=\tfrac14(1-z_2)+\tfrac12\[\(\tfrac34\)^2-\left|z-\(0,\tfrac14,0,...,0\)\right|^2\]\\\label{eq:distances}
&>\tfrac14(1-z_2)=\tfrac14 \d_{H}(z).
\end{align}
Our goal in this part of the proof is to show that $p_{J_{1,L}}(s,x,y)\gtrsim p_{R^{(2)}_L}(s,x,y)$ holds for $y\in \mathbb T\cap\{y_1\leq 0\}$ (keeping in mind the special form of $y=(y_1,y_2,0,...,0)$) and $0<s<1$. Since  $J_{1,L}\cap \{x\in \R^n:x_1>0\}=R^{(1)}_L\cap \{x\in \R^n:x_1>0\}$, then $\{s<\tau_{J_{1,L}};s\geq\tau_{1}\}=\{s<\tau_{R^{(1)}_L};s\geq\tau_{1}\}$ so we may change the condition $s<\tau_{J_{1,L}}$ into $s<\tau_{R^{(1)}_L}$ in \eqref{eq:shortMarkov}. Furthermore, adding also the condition $W(\tau_1)\in \{0\}\times I_0$ under the expectation we arrive at  
\begin{align}\label{eq:Hunt2}
&p_{J_{1,L}}(s,x,y)\geq\E^x\[s<\tau_{R^{(1)}_L};s\geq\tau_{1};W(\tau_1)\in \{0\}\times I_0;p_{J_{1,L}}(s-\tau_{1},W(\tau_{1}),y)\].
\end{align}
Observe now  that the assumption $y=(y_1,y_2,0,...,0)\in  {\mathbb T}$ and the condition $W(\tau_1)\in \{0\}\times I_0$ imply  $y,W(\tau_1)\in  B_0$.
Consequently, by \eqref{eq:D1<D2}, (\ref{eq:MS}),  \eqref{eq:dBdR} and \eqref{eq:estH} we get for $0<s\leq1$ 
\begin{align*}
p_{J_{1,L}}(s,W(\tau_1),y)&\geq p_{B(0,1)}(s,W(\tau_1),y)\\
&\gtrsim\(1\wedge\frac{ \delta_{H}(W(\tau_1))\delta_{H}(y)}s\)p(s,W(\tau_1),y)\\
&\approx p_H(s,W(\tau_1),y)\geq p_{R_L^{(1)}}(s,W(\tau_1),y).
\end{align*}
Applying this to \eqref{eq:Hunt2}, we obtain 
\begin{align*}
&p_{J_{1,L}}(s,x,y)\gtrsim\E^x\[s<\tau_{R^{(1)}_L};s\geq\tau_{1};W(\tau_1)\in \{0\}\times I_0;p_{R_L^{(1)}}(s-\tau_{1},W(\tau_{1}),y)\].
\end{align*}
The inclusion $R^{(2)}_L\subset R^{(1)}_L$ imples $\{s<\tau_{R^{(2)}_L}\}\subset \{s<\tau_{R^{(1)}_L}\}$ and $p_{R^{(1)}_L}(s-\tau_{1},W(\tau_{1}),y)\geq p_{R^{(2)}_L}(s-\tau_{1},W(\tau_{1}),y)$.  Hence
\begin{align*}
&p_{J_{1,L}}(s,x,y)\gtrsim \E^x\[s<\tau_{R^{(2)}_L};s\geq\tau_{1};W(\tau_1)\in \{0\}\times I_0;p_{R^{(2)}_L}(s-\tau_{1},W(\tau_{1}),y)\].
\end{align*}
Furthermore,  the condition $W(\tau_1)\in \{0\}\times I_0$ is always satisfied on the set $\{s<\tau_{R^{(2)}_L}\}$, so it may be removed. Thus, repeating argument from \eqref{eq:longMarkov}, we conclude
\begin{align*}
p_{J_{1,L}}(s,x,y)&\gtrsim \E^x\[s<\tau_{R^{(2)}_L};s\geq\tau_{1};p_{R^{(2)}_L}(s-\tau_1,W(\tau_1),y)\]\\
&=p_{R^{(2)}_L}(s,x,y),
\end{align*}
as required. Next we will show that $p_{R^{(2)}_L}(s,x,y)$ admits  lower estimate from \eqref{eq:S1L}.
Cylindrical form of $R^{(2)}_L$ combined with \eqref{eq:estH} and \eqref{eq:MS} give us
\begin{align*}
p_{R^{(2)}_L}(s,x,y)&=p_{(-\infty,L)}(s,x_1,y_1)p_{I_0}(s,\widetilde x, \widetilde y)\\
&\gtrsim \(1\wedge \frac{(L-x_1)(L-y_1)}{s}\)\(1\wedge\frac{\delta_{I_0}(\widetilde x)\delta_{I_0}(\widetilde y)}{s}\)p(s,x,y)\\
&= \(1\wedge\frac{\delta_{I_0}(\widetilde y)}{s}\)p(s,x,y).
\end{align*}
Since $y=(y_1,y_2,0,...,0)\in \mathbb T$, we have $y\in B_n((0,\frac34,0,...,0),\frac14)$, and, analogously as in \eqref{eq:distances} (or by rescaling), one can show that $\delta_{I_0}(\widetilde y)\approx \delta_{J_{1,L}}(y)$, which completes the proof for $y\in \widetilde{\mathbb T}$.

Finally, let us consider any $y\in J_{1,L}$ such that $y_1\leq0$, as in the assertion.  Denote by $m=(m_1,...,m_n)$ the point on the interval $\overline{xy}$ such that $m_1=0$. Due to the special form of $x$ and $y$ we have $ m=(0,m_2, 0,...,0)$ with $m_2\in[0,1)$. Next, we put $\alpha_{xy}=\frac{|m-y|}{|x-y|}$ and  denote 
$$ m'=m-\tfrac12\sqrt{\alpha_{xy}s}(0,1,0,...,0)=(0,m_2-\tfrac12\sqrt{\alpha_{xy}s}, 0,...,0).$$
Since  $\alpha_{xy}<\frac1{L-1}<\frac12$, we have  $B( m',\frac14\sqrt{\alpha_{xy}s})\subset \widetilde {\mathbb T}$, and  for all $z\in B(m',\frac14\sqrt{\alpha_{xy}s})$ it holds
\begin{align*}
\delta_{J_{1,L}}(z)&\geq \delta_{J_{1,L}}(m')-\frac14\sqrt{\alpha_{xy}s}\\
&\geq\(1-(m'_2+\tfrac14\sqrt{\alpha_{xy}s})\)\wedge \((m'_2-\tfrac14\sqrt{\alpha_{xy}s})-(-1)\)\\
&\geq  \(1-m_2+\tfrac14\sqrt{\alpha_{xy}s}\)\wedge \(1-\tfrac34\sqrt{\alpha_{xy}s}\)\\
&\geq  \(\delta_{J_{1,L}}(m)+\tfrac14\sqrt{\alpha_{xy}s}\)\wedge \(\tfrac14\)\\
&\geq\frac18\(\delta_{J_{1,L}}(m)+\tfrac14\sqrt{\alpha_{xy}s}\).
\end{align*}
Furthermore, we clearly have $\delta_{J_{1,L}}(m)\geq \delta_{J_{1,L}}(y)$ and, by intercept theorem, $\delta_{J_{1,L}}(m)\geq \alpha_{xy}\delta_{J_{1,L}}(x)=\alpha_{xy}$, which eventually gives us 
$$\delta_{J_{1,L}}(z)\gtrsim w(x,y):=  \sqrt{\alpha_{xy}s}+\delta_{J_{1,L}}(y)+\alpha_{xy},\ \ \ \ \ z\in B(m',\frac14\sqrt{\alpha_{xy}s}).$$
Consequently, from previous case we have for $z\in \mathbb{ \widetilde{ T}}\cap \{z_1<0\}$ 
\begin{align*}
p_{J_{1,L}}((1-\alpha_{xy})s,x,z)&\gtrsim p(t,x,z)\(1\wedge\frac{\delta_{J_{1,L}}(z)}{(1-\alpha_{xy})s}\)\gtrsim p(t,x,z)\(1\wedge\frac{w(x,y)}s\).
\end{align*}
Additionally,  for the same range of $z$, \eqref{eq:MS} implies
\begin{align*}
p_{J_{1,L}}(\alpha_{xy}s,z,y)&\geq p_{B(0,1)}(\alpha_{xy}s,z,y)\gtrsim  p(\alpha_{xy} s,z,y)\(1\wedge\frac{\delta_{J_{1,L}}(y)w(x,y)}{\alpha_{xy} s}\).
\end{align*}
Then, by Chapman-Kolmogorov identity and Corollary \ref{cor:CKlow}, we get
\begin{align*}
p_{J_{1,L}}(s,x,y)&\geq \int_{B( m',\sqrt{\alpha_{xy}s})\cap \{z_1<0\}}p_{J_{1,L}}((1-\alpha_{xy})s,x,z)p_{J_{1,L}}(\alpha_{xy}s,z,y)dz\\
&\geq \(1\wedge\frac{w(x,y)}s\)\(1\wedge\frac{\delta_{J_{1,L}}(y)w(x,y)}{\alpha_{xy} s}\)\\
&\ \ \ \ \  \int_{B( m'',\frac12\sqrt{\alpha_{xy}s})}p((1-\alpha_{xy})s,x,z)p(\alpha_{xy}s,z,y)dz\\
&\gtrsim\(1\wedge\frac{w(x,y)}s\)\(1\wedge\frac{\delta_{J_{1,L}}(y)w(x,y)}{\alpha_{xy} s}\)p(s,x,y),
\end{align*} 
where
$$ m''=m'-\tfrac12\sqrt{\alpha_{xy}s}(1,0,...,0)=(-\tfrac12\sqrt{\alpha_{xy}s},m_2-\tfrac12\sqrt{\alpha_{xy}s}, 0,...,0).$$
In order to finish the proof, we need to show that the above product of two minima is greater (up to a constant factor) than $\(1\wedge\frac{\delta_{J_{1,L}(y)}}{s}\)$. If the right-hand side minimum is equal to $1$, then it is enough to use the bound $w(x,y)> \delta_{J_{1,L}}(y) $. If both of them are smaller than $1$,  one can use $w(x,y)> \sqrt{\alpha_{xy}s} $. Finally, if only the right-hand side one is smaller than $1$, then we need to employ the bound $w(x,y)>\alpha_{xy}$. The proof is complete.
 \end{proof}

The next theorem provides a general lower bound of the form as in \eqref{eq:Z} but with proper exponential behavior. Nevertheless, both results concern different classes of sets (with non-empty intersection).

\begin{thm}\label{thm:lower} For any convex set $D\in \mathcal C^{1,1}_r$ and $T>0$ there is  $C=C(n,r,T)$ such that  
$$p_{D}(t,x,y)\geq C\,\(1\wedge \frac{\delta_{D}(x)\delta_D(y)}{t}\)p(t,x,y).$$
\end{thm}

\begin{proof}
Assume $t<1$. First, consider the case  $\delta_D(x)>10\sqrt t$. Due to the bound  \eqref{eq:vdB}, we may additionally assume $\delta_D(y)\leq \sqrt t$. Put $\alpha_{xy}=5\sqrt t/\delta_{D}(x)<1/2$ and  $m=\alpha_{xy}x+(1-\alpha_{xy})y$. It follows from  intercept theorem that $\delta_D(m)>5\sqrt t$. Thus, for every $z\in B(m,\sqrt{\alpha_{xy}t})$ we have  $\delta_D(z)\geq 4\sqrt t$ and $|z-y|\geq \delta_D(z)-\delta_D(y)\geq 3\sqrt t$, which lets us employ Lemma \ref{lem:S1L}. Precisely, for a given $z$ one can transform isometrically the set $J_{\sqrt t,L_z}$ into $J'_{\sqrt t,L_z}$, for suitably chosen $L=L(z)>3\sqrt t$, such that $\delta_D(y)=\delta_{J'_{\sqrt t, L_z}}(y)$ and the point $(L_z-\sqrt t,0,0,...,0)$ is transformed into $z$. Then, Lemma \ref{lem:S1L} gives us for $z\in B(m,\sqrt{\alpha_{xy}t})$
\begin{align*}
p_D(\alpha_{xy}t,z,y)&\geq p_{J'_{\sqrt t, L_z}}(\alpha_{xy}t,z,y)\\
& \gtrsim \(1\wedge\frac{\delta_D(y)\sqrt t}{\alpha_{xy}t}\)p(\alpha_{xy}t,z,y)\\
&=\(1\wedge\frac{\delta_D(y)\delta_D(x)}{6t}\)p(\alpha_{xy}t,z,y),
\end{align*}
while, by \eqref{eq:vdB}, we have
$$p_D((1-\alpha_{xy})t,x,z)\gtrsim p((1-\alpha_{xy})t,x,z).$$
Consequently, by Chapman-Kolmogorov identity and Corollary \ref{cor:CKlow} we get
\begin{align*}
p_D(t,x,y)&\geq\int_{B(m,\sqrt{\alpha_{xy}t})}p_D((1-\alpha_{xy})t,x,z)p_D(\alpha_{xy}t,z,y)dz\\
&\gtrsim\(1\wedge\frac{\delta_D(y)\delta_D(x)}{t}\)\int_{B(m,\sqrt{\alpha_{xy}t})}p((1-\alpha_{xy})t,x,z)p(\alpha_{xy}t,z,y)dz\\
&\gtrsim\(1\wedge\frac{\delta_D(y)\delta_D(x)}{t}\)p(t,x,y).
\end{align*}

Consider now $\delta_D(x), \delta_D(y)\leq 10\sqrt t$. Let $p_{xy}$ be a point such that $|p_{xy}-\frac{x+y}{2}|=11\sqrt {t/2}$ and $\delta_D(p_{xy})\geq 11\sqrt {t/2}$. Such point exists for sufficiently small $T$. Then, for $z\in B(p_{xy},\sqrt {t/2})$, we have $\delta_D(z)\geq 10\sqrt{t/2}$ and, from the previous case,  assertion is true for $p_D(t/2,x,z)$ and $p_D(t/2,z,y)$.  
Thus, by Corollary \ref{cor:CKlow} we get
\begin{align*}
p_D(t,x,y)&\geq\int_{B(p_{xy},\sqrt{t/2})}p_D(t/2,x,z)p_D(t/2,z,y)dz\\
&\gtrsim\frac{\delta_D(x)\sqrt t}{t}\frac{\delta_D(y)\sqrt t}{t}\int_{B(p_{xy},\sqrt{t/2})}p(t/2,x,z)p(t/2,z,y)dz\\
&\approx\(1\wedge\frac{\delta_D(y)\delta_D(x)}{t}\)p(t,x,y),
\end{align*}
where we used $\delta_D(x)\approx \delta_D(y)\approx \sqrt t$. Finally, let us observe that the range of $T$ may be easily extended, but the cost we pay is  decrease of the constant $C$. Indeed, for $T\leq t<2T$ and $x,y\in D$ there is a point  $q_{xy}\in D$  such that $\left|q_{xt}-\frac{x+y}2\right|\leq \sqrt t \frac{r}{2\sqrt T}$ and $\delta_D(q_{xy})\geq  \sqrt t \frac{r}{2\sqrt T}$. Then, repeating previous arguments,  we obtain
\begin{align*}
p_D(t,x,y)&\geq\int_{B\(q_{xy},\sqrt t \frac{r}{4\sqrt T}\)}p_D(t/2,x,z)p_D(t/2,z,y)dz\\
&\gtrsim\(1\wedge\frac{\delta_D(x)\sqrt t\frac{r}{2\sqrt T}}{t}\)\(1\wedge\frac{\delta_D(y)\sqrt t\frac{r}{2\sqrt T}}{t}\)\int_{B\(q_{xy},\sqrt t \frac{r}{4\sqrt T}\)}p(t/2,x,z)p(t/2,z,y)dz\\
&\geq c\(1\wedge\frac{\delta_D(y)\delta_D(x)}{t}\)p(t,x,y),
\end{align*}
for some constant $c=c(n,r,T)$. 
\end{proof}

The lower bound from Theorem \ref{thm:lower} may be improved by suitable  application of Chapman-Kolmogorov identity, as presented below. In fact, the procedure could be iterated, however, further iterations lead to much more complicated forms, which do not seem to be a relevant enhancement in the general case.

\begin{thm}\label{cor:mainlower}
For any convex set $D\in \mathcal C^{1,1}_r$ and $T>0$ there is  $C=C(n,r,T)$ such that  
\begin{align}\nonumber
&p_D(t,x,y)\\\label{eq:lower1}
&\geq C p(t,x,y)\(1\wedge \frac{\d_{D}(x)\(\d_{D}\(\frac{x+y}2\)+\sqrt t\)}{t}\)\(1\wedge \frac{\d_{D}(y)\(\delta_{D}\(\frac{x+y}{2}\)+\sqrt t\)}{t}\)\\\label{eq:lower2}
&\approx  C p(t,x,y)\[\(1\wedge\frac{\delta_D(x)\delta_D(y)}{t}\)+\(1\wedge \frac{\d_{D}(x)\d_{D}\(\frac{x+y}2\)}{t}\)\(1\wedge \frac{\d_{D}(y)\delta_{D}\(\frac{x+y}{2}\)}{t}\)\].
\end{align}
\end{thm}
\begin{proof}
For $\delta_D\(\frac{x+y}{2}\)\leq2\sqrt t$ the inequality \eqref{eq:lower1} follows directly from Theorem \ref{thm:lower}. We therefore assume $\delta_D\(\frac{x+y}{2}\)\geq2\sqrt t$ .  Then, for  $z\in B\(\frac{x+y}2,\sqrt t\)\subset D$ it holds
 $$\delta_D(z)\geq \frac12\delta_D\(\frac{x+y}{2}\)\geq \frac14\(\delta_D\(\frac{x+y}{2}\)+\sqrt t\).$$
Hence, by  Chapman-Kolmogorov equality,  Theorem \ref{thm:lower} and Corollary \ref{cor:CKlow} we get
\begin{align*}
&p_D(t,x,y)\\
&=\int_Dp_D(t/2,x,z)p_D(t/2,z,y)dz\\
&\geq C\int_{B\(\frac{x+y}2,\sqrt t\)}p(t/2,x,z)\(1\wedge \frac{\d_D(x)\d_D(z)}{t}\)p(t/2,z,y)\(1\wedge \frac{\d_D(z)\d_D(y)}{t}\)dz\\
&\geq C\(1\wedge \frac{\d_{D}(x)\(\d_{D}\(\frac{x+y}2\)+\sqrt t\)}{t}\)\(1\wedge \frac{\d_{D}(y)\(\delta_{D}\(\frac{x+y}{2}\)+\sqrt t\)}{t}\)\\
&\ \ \ \ \ \ \ \  \times \int_{B\(\frac{x+y}2,\sqrt t\)}p(t/2,x,z)p(t/2,z,y)dz\\
&\gtrsim C\(1\wedge \frac{\d_{D}(x)\(\d_{D}\(\frac{x+y}2\)+\sqrt t\)}{t}\)\(1\wedge \frac{\d_{D}(y)\(\delta_{D}\(\frac{x+y}{2}\)+\sqrt t\)}{t}\)p(t,x,y),
\end{align*}
as required. In order to show \eqref{eq:lower2} we consider two cases as well. If $\delta_{D}\(\frac{x+y}{2}\)\geq \sqrt t$, 
\begin{align*}
&\(1\wedge \frac{\d_{D}(x)\(\d_{D}\(\frac{x+y}2\)+\sqrt t\)}{t}\)\(1\wedge \frac{\d_{D}(y)\(\delta_{D}\(\frac{x+y}{2}\)+\sqrt t\)}{t}\)\\
&\approx \(1\wedge \frac{\d_{D}(x)\d_{D}\(\frac{x+y}2\)}{t}\)\(1\wedge \frac{\d_{D}(y)\delta_{D}\(\frac{x+y}{2}\)}{t}\)\\
&\geq \(1\wedge \frac{\d_{D}(x)}{\sqrt t}\)\(1\wedge \frac{\d_{D}(y)}{\sqrt t}\)\geq \(1\wedge \frac{\delta_D(x)\delta_D(y)}{t}\),
\end{align*}
which means that the right-hand side term in brackets in \eqref{eq:lower2} is dominating and comparable with the  factor from \eqref{eq:lower1}. Similarly, for  $\delta_{D}\(\frac{x+y}{2}\)\leq \sqrt t$ we conclude $\delta_D(x),\delta_D(y)\leq 2\sqrt t$, and hence
\begin{align*}
&\(1\wedge \frac{\d_{D}(x)\(\d_{D}\(\frac{x+y}2\)+\sqrt t\)}{t}\)\(1\wedge \frac{\d_{D}(y)\(\delta_{D}\(\frac{x+y}{2}\)+\sqrt t\)}{t}\)\\
&\approx  \frac{\delta_D(x)\delta_D(y)}{t} \approx  \(1\wedge \frac{\delta_D(x)\delta_D(y)}{t}\)\geq \(1\wedge \frac{\delta_D(x)\delta_D(y)\(\delta_D\(\frac{x+y}2\)\)^2}{t^2}\)\\
&\approx \frac{\delta_D(x)\delta_D(y)\(\delta_D\(\frac{x+y}2\)\)^2}{t^2}\approx \(1\wedge \frac{\d_{D}(x)\d_{D}\(\frac{x+y}2\)}{t}\)\(1\wedge \frac{\d_{D}(y)\delta_{D}\(\frac{x+y}{2}\)}{t}\),
\end{align*}
which ends the proof.
\end{proof}

\section{Two-sided estimates}
\subsection{General results}

For a strictly convex $\mathcal C^{1,1} $ domain $D$   we define 
\begin{align*}
Q_{D}&:=\inf_{w,z\in\partial D, w\neq z}\frac{\delta_D\(\frac{w+z}{2}\)}{\delta_{H_{w}}\(\frac{w+z}2\)},\\
R_{D}&:=\min\left\{\inf_{\substack{w,z\in\partial D, w\neq z\\\delta_D\(\frac{w+z}2\)\leq1}}\frac{\delta_D\(\frac{w+z}{2}\)}{\delta_{H_{w}}\(\frac{w+z}2\)},\inf_{\substack{w,z\in\partial D, w\neq z\\\delta_D\(\frac{w+z}2\)>1}}\sup_{\substack{m\in\overline{wz}\\\delta_D(m)>1}}\frac{\delta_D\(m\)}{\delta_{H_{w}}\(m\)}\right\}.
\end{align*}
Note that since $D$ is a $\mathcal C^{1,1} $ domain, the half-space $H_w$ is well defined for any $w\in\partial D$. It is clear that  $0\leq Q_D,R_D\leq 1$. Furthermore, since $\frac{w+z}2$ is a possible value of  $m$ in the supremum in the definition of $R_{D}$, it holds $Q_D\leq R_D$. In general, we will be expecting $Q_D, R_D>0$. The condition $Q_D>0$ means that for any $w,z\in \partial D$ the distance from the midpoint $\frac{w+z}2$ to the boundary $\partial D$ is comparable with the distances to $P_{w}$ and $P_{z}$.  In case $R_D>0$ the condition is weaker whenever $\delta_D\(\frac{w+z}2\)>1$, as we only require  existence of a point at the interval $\overline{wz}$ whose distance to $\partial D$ is greater than $1$ and comparable to distance to $P_{w}$. Let us introduce the following two class of sets corresponding to the characteristics $Q_D$ and $R_D$
\begin{align*}
\mathcal S_{Q}&:=\left\{D\in \mathcal C^{1,1}(\R^n): D\text{ is strictly convex}, Q_D>0\right\},\\
\mathcal S_{R}&:=\left\{D\in \mathcal C_r^{1,1}(\R^n) \text{ for some } r>0: D\text{ is strictly convex}, R_D>0\right\}.
\end{align*}
In the definition of $\mathcal S_Q$ we do not require the sets to be in $\mathcal C_r^{1,1}(\R^n)$, since it turns out that every $D\in \mathcal S_Q$ is bounded (see Lemma \ref{lem:dH<dD}) and it is well known that every bounded $\mathcal C^{1,1}$ set belongs to $\mathcal C_r^{1,1}(\R^n)$ for some $r>0$. Both of the  classes $\mathcal S_{Q}$ and $\mathcal S_R$  contain nontrivial and important examples (see Propositions \ref{prop:y>x^p} and \ref{prop:analytical}). It seems also not easy to construct a  strictly convex  $\mathcal C^{1,1}_r$ set which does not belong to $\mathcal S_R$.  

The following monotonicity property will be needed in the sequel:
\begin{prop}\label{prop:mono} Let $D$ by a $\mathcal C^{1,1}$ convex domain in $\R^n$, $n\geq 2$, and let $w,z\in \partial D$. Then the function 
$$\alpha\longrightarrow\frac{\delta_D\((1-\alpha)w+\alpha z\)}{\delta_{H_{w}}\((1-\alpha)w+\alpha z\)}$$   
is non-increasing on $[0,1]$.
\end{prop}
\begin{proof}
Consider first $n=2$. Let $p,q\in\overline{wz}$  be such that $0<|w-p|<|w-q|\leq |w-z|$ and let $p'\in\partial D$ be a point realising the distance of $p$ to $\partial D$. Then $P_{p}:=P_{p'}$ is a line tangent to $D$ at $p'$, and $H_p$ is the related half-plane.
 Thus, intercept theorem gives us
\begin{align*}
 \frac{\delta_{H_w}(p)}{|w-p|}= \frac{\delta_{H_w}(q)}{|w-q|},\ \ \ \ \text{ and }\ \ \ \ \frac{\delta_D(p)}{|w-p|}\geq \frac{\delta_{H_p}(q)}{|w-q|}.
\end{align*}
This implies
$$\frac{\delta_D(p)}{\delta_{H_w}(p)}\geq \frac{\delta_{H_p}(q)}{\delta_{H_w}(q)}\geq \frac{\delta_{D}(q)}{\delta_{H_w}(q)},$$
as required. For $n\geq 3$ let $P$ be the $2$-dimensional plane containing the interval $\overline {wz}$ and the point $p'\in\partial D$  realising the distance of $p$ to $\partial D$. Since $\delta_{D\cap P}(p)=\delta_{D}(p)$, $\delta_{D\cap P}(q)\geq \delta_{D}(q)$ and, by intercept theorem,  $\frac{\delta_{H_w\cap P}(p)}{\delta_{H_w\cap P}(q)}=\frac{\delta_{H_w}(p)}{\delta_{H_w}(q)}$, we may apply the result for $n=2$ and get
$$\frac{\delta_D(p)}{\delta_{H_w}(p)}= \frac{\delta_{D\cap P}(p)}{\delta_{H_w}(p)}\geq \frac{\delta_{D\cap P}(q)\frac{\delta_{H_w\cap P}(p)}{\delta_{H_w\cap P}(q)}}{\delta_{H_w}(p)}\geq \frac{\delta_{D}(q)\frac{\delta_{H_w\cap P}(p)}{\delta_{H_w\cap P}(q)}}{\delta_{H_w}(p)}=\frac{\delta_{D}(q)}{\delta_{H_w}(q)},$$
which ends the proof.
\end{proof}

The next lemma shows that if $Q_D>0$,  then $D$ is bounded and the the infimum from the definition of $Q_D$ taken over all $z,w\in \bar D$ is positive as well, and consequently $\delta_D\(\frac{w+z}2\)$ and  $\delta_{H_z}\(\frac{w+z}2\)$ (for any choice of $H_w$) are comparable for any $z,w\in \bar D$.
\begin{lem}\label{lem:dH<dD} Let $D\in \mathcal S_Q$. Then
\begin{itemize}
\item[i)] $D$ is bounded,
\item [ii)] for every $x,y\in  D$ we have
\begin{align}\label{eq:dH<dD}
\delta_{D}\(\frac{x+y}2\)\leq \delta_{H_x}\(\frac{x+y}2\)\leq \frac3{Q_D} \delta_{D}\(\frac{x+y}{2}\),
\end{align}
where $H_x$ is any half-space such that $\delta_D(x)=\delta_{H_x}(x)$ and $D\subset H_x$.
\end{itemize} 
\end{lem}
\begin{proof}
$i)$ Assume $D$ is an unbounded strictly convex $\mathcal C^{1,1}$ set with $Q_D>0$. There exists a half-line $l$ starting at some point $w\in \partial D$ and contained in $\overline{D}$. Let $P$ be a ($2$-dimensional) plane containing $l$ and equipped with coordinate system of axes $Ox$ and $Oy$ such that $l$ is the nonnegative half-line of $Oy$.  The intersection $P\cap \partial D$ is then a graph of a strictly convex $\mathcal C^{1,1}$ function $f(x)$.  Without loss of the generality we may  assume that $f$ is increasing on $[0,\infty)$. The distance of any point from $P$ of the form $(x,f(x))$ to the hyper-plane $H_w$ is proportional to its distance to a line in $P$ tangent to $f$ at $(0,0)$ (i.e. the point $w$). Hence, since $f$ is strictly convex and increasing on $[0,\infty)$, there is a constant $c_1>0$ such that $\delta_{H_w}((x,f(x)))>c_1 f(x)$ for $x>1$. Additionally, for $y>f(x)$ we clearly have $\delta_D((x,y))\leq\delta _{D\cap P}\((x,y)\)\leq \(y-f(x)\)\wedge\(f^{-1}(y)-x\)$, where $f^{-1}$ is inverse of $f$ on $[0,\infty)$ and $(x,y)$ is a point on $P\subset \R^n$ in coordinated introduced on $P$. Taking $z=(x,f(x))\in \partial D$, $x>1$, we get
$$\frac{\delta_D\(\frac{w+z}{2}\)}{\delta_{H_{w}}\(\frac{w+z}2\)}\leq \frac{\(\frac12f(x)-f\(\frac12x\)\)\wedge\(\frac12x\)}{\frac12c_1f(x)}=:g(x).$$
We will show that $g(x)$ tends to zero as $x\rightarrow\infty$, which contradicts the assumption $Q_D>0$. From strict convexity of $f$ we have $\lim _{x\rightarrow\infty}f(x)=\infty$ and $f'(x), \frac12f(x)-f\(\frac12x\)$ are increasing for $x\geq0$. If $\frac12f(x)-f\(\frac12x\)$ is bounded, then $g(x)$ clearly tends to zero. In the other case, it is when $\frac12f(x)-f\(\frac12x\)$ tends to infinity, we employ L'H\^opital's rule and obtain
$$\limsup_{x\rightarrow\infty}\frac{\delta_D\(\frac{w+z}{2}\)}{\delta_{H_{w}}\(\frac{w+z}2\)}\leq \lim_{x\rightarrow\infty}\frac{\(f'(x)-f'\(\frac12x\)\)\wedge1}{c_1f'(x)}.$$
Since $f'(x)$ tends either to a constant or to infinity, the last limit equals zero, as required.

$ii)$ The first inequality follows simply from the inclusion $D\subset H_x$. Let $x', y'\in \partial D$ be points realising distances of $x$ and $y$, respectively, to the boundary $\partial D$, i.e. such that  $|x-x'|=\delta_D(x)$ and $|y-y'|=\delta_D(y)$, and denote $H_x=H_{x'}$, $H_y=H_{y'}$. If $x'=y'$, then $\d_{H_x}\(\frac{x+y}2\)=\d_{D}\(\frac{x+y}2\)$ and \eqref{eq:dH<dD} holds since $Q_D\leq1$. In case $x'\neq y'$ we observe
\begin{align*}
\left|\frac{x+y}2-\frac{x'+y'}2\right|&\leq\frac12|x-x'|+\frac12|y-y'|=\frac12\(\delta_D(x)+\delta_D(y)\)\leq \delta_{D}\(\frac{x+y}{2}\),
\end{align*}
where the last inequality follows from convexity of $D$. Consequently, 
\begin{align*}
\delta_{H_x}\(\frac{x+y}{2}\)&\leq \delta_{H_x}\(\frac{x'+y'}{2}\)+\left|\frac{x+y}2-\frac{x'+y'}2\right|\\
&\leq\frac1{Q_D} \delta_{D}\(\frac{x'+y'}{2}\)+\delta_{D}\(\frac{x+y}{2}\)\\
&\leq \frac1{Q_D}\( \delta_{D}\(\frac{x+y}{2}\)+\left|\frac{x+y}2-\frac{x'+y'}2\right|\)+\delta_{D}\(\frac{x+y}{2}\)\\
&\leq \(\frac2{Q_D}+1\) \delta_{D}\(\frac{x+y}{2}\)\leq \frac3{Q_D}\delta_{D}\(\frac{x+y}{2}\),
\end{align*}
where we used the inequality $Q_D\leq 1$.
\end{proof}
Theorem \ref{cor:mainlower} and Lemma \ref{lem:dH<dD}$ii)$ applied to   Corollary  \ref{cor:mainupper} follow directly
\begin{cor}\label{cor:SQest}
If  $D\in \mathcal S_Q$   then
\begin{align*}
&p_D(t,x,y)\\
&\stackrel{r,Q_D,T}{\approx}  p(t,x,y)\[\(1\wedge\frac{\delta_D(x)\delta_D(y)}{t}\)+\(1\wedge \frac{\d_{D}(x)\d_{D}\(\frac{x+y}2\)}{t}\)\(1\wedge \frac{\d_{D}(y)\delta_{D}\(\frac{x+y}{2}\)}{t}\)\]\\
&\stackrel{\phantom{n,r,}Q_D\phantom{,T}}{\approx}  p(t,x,y)\[\(1\wedge\frac{\delta_D(x)\delta_{D}(y)}{t}\)+\(1\wedge \frac{\d_{H_x}(x)\d_{H_x}\(y\)}{t}\)\(1\wedge \frac{\d_{H_y}(y)\delta_{H_y}\(x\)}{t}\)\].
\end{align*}
holds for $x,y\in D$, $0<t<T$.
\end{cor}

Moreover, it turns out  that $\mathcal S_Q$ is the exact subclass of $\mathcal C^{1,1}$ domains for which the lower bound from Theorem \ref{cor:mainlower} is equivalent (up to a multiplicative constant) to the upper bound.

\begin{thm}\label{thm:SQ}
Let $D$ be a strictly convex $\mathcal C^{1,1}$ set. Then   $D\in \mathcal S_Q$   if and only if  
\begin{align}\label{eq:2sided}
&p_D(t,x,y)\\\nonumber
&\stackrel{D,T}\approx  p(t,x,y)\[\(1\wedge\frac{\delta_D(x)\delta_D(y)}{t}\)+\(1\wedge \frac{\d_{D}(x)\d_{D}\(\frac{x+y}2\)}{t}\)\(1\wedge \frac{\d_{D}(y)\delta_{D}\(\frac{x+y}{2}\)}{t}\)\].
\end{align}
holds for $x,y\in D$, $0<t<T$.
\end{thm}
\begin{proof}
If $Q_D>0$, the estimate \eqref{eq:2sided} follows from Corollary \ref{cor:SQest}. 

Let us  assume that \eqref{eq:2sided} holds for all $x,y\in D$ and $0<t<T$ for some $T>0$, and consider $w,z\in\partial D$ such that $w\neq z$. Since $D$ is a $\mathcal C^{1,1}$ set, there is a ball of radius $r>0$ contained in $D$ and tangent to it at $w$, which ensures existence of a point $m$ from the interval $\overline{wz}$ such that $|m-w|<|w-z|/4$ and
\begin{align}\label{eq:dm>dHz}
\delta_D\(m\)>\frac12\delta_{H_{w}}(m).
\end{align}
Set
\begin{align}\label{eq:t=min}
\sqrt t=\min\left\{r,\frac16\delta_D\(m\), \sqrt T \right\},
\end{align}
and let $m'$ be a point at the interval $\overline{wm}$ such that $\delta_D(m')=6\sqrt t$. Then, from \eqref{eq:dm>dHz} and Proposition \ref{prop:mono} we have
 \begin{align}\label{eq:dm'}
\delta_D\(m'\)\approx  \delta_{H_w}(m')\approx \sqrt t.
\end{align}
Since we are going to apply  estimates from \eqref{eq:2sided} and  $w\notin D$,  we approximate $w$  by a point from $D$. Precisely, let $x$  be a point from $\overline{wm'}$ such that $|w-x|<\sqrt t$ and let us put $\alpha:=\frac{|x-m'|}{\left|x-\frac{w+z}2\right|}<\frac12$.
Then, Chapman-Kolmogorov identity  gives us
\begin{align}\label{eq:C-K5}
p_D\(t,x,\frac{w+z}2\)&\geq \int_{B(m',\sqrt{\alpha t})}p_D(\alpha\,t,x,v)p_D\((1-\alpha)t,v,\frac{w+z}2\)dv.
\end{align}
For $v\in B(m',\sqrt{\alpha t})$ we have $|x-v|\geq 4\sqrt t$ and $\delta_D(v)\geq\frac56 \delta_D(m')\geq 5\sqrt t $, which allows us  to employ Lemma \ref{lem:S1L} (in the same manner as in  the proof of Theorem \ref{thm:lower}) and get
$$p_D(\alpha t,x,v)\gtrsim \(1\wedge\frac{\delta_D(x)\sqrt t}{\alpha t}\)p(\alpha t,x,v)\approx \(1\wedge\frac{\delta_D(x)\delta_D(m')}{\alpha t}\)p(\alpha t,x,v).$$
Note that we could have not used Theorem  \ref{thm:lower}, since we  assume $D$ to be any strictly convex $C^{1,1}$ set, so it may not belong to any $\mathcal C^{1,1}_r$, $r>0$. Next, intercept theorem and the inequality \eqref{eq:dm>dHz} give us 
\begin{align*}
\delta_D(m')&\gtrsim\delta_{H_{w}}(m')
=2\frac{|w-m'|}{|w-z|}\delta_{H_{w}}\(\frac{w+z}2\)\\
&\geq \frac{|x-m'|}{2\left|x-\frac{w+z}2\right|}\delta_{H_{w}}\(\frac{w+z}2\)=\frac12\alpha\, \delta_{H_{w}}\(\frac{w+z}2\),
\end{align*}
which follows
$$p_D(\alpha t,x,w)\gtrsim\(1\wedge\frac{\delta_D(x)\delta_{H_{w}}\(\frac{w+z}2\)}{ t}\)p(\alpha t,x,w).$$
Furthermore, by \eqref{eq:vdB}, we have
\begin{align*}
p_D\((1-\alpha)t,v,\frac{w+z}2\)\gtrsim p\((1-\alpha)t,v,\frac{w+z}2\).
\end{align*}
Applying the last two bounds to \eqref{eq:C-K5} and using Corollary \ref{cor:CKlow}, we conclude
\begin{align}\label{eq:pD>c_3}
p_D\(t,x,\frac{w+z}2\)&\geq c_1\(1\wedge\frac{\delta_D(x)\delta_{H_{w}}\(\frac{w+z}2\)}{ t}\) p\(t,x,\frac{w+z}2\),
\end{align}
for some $c_1>0$. On the other hand, convexity of $D$ implies
 \begin{align}\label{eq:supdD}
 \delta_D\(\frac{w+z}2\)\geq \frac12 \sup_{v\in\overline{wz}}\{\delta_D\(v\)\},
 \end{align} and consequently, by \eqref{eq:2sided},
 \begin{align}\nonumber
 p_D\(t,x,\frac{w+z}2\)&\leq c_2 p\(t,x,\frac{w+z}2\)\\\nonumber
 &\ \ \ \times\[\(1\wedge\frac{\delta_D(x)\delta_{D}\(\frac{w+z}2\)}{ t}\)+\(1\wedge\frac{\delta_D(x)\delta_{D}\(\frac12 x+\frac14w+\frac14 z\)}{ t}\)\]\\
 \label{eq:pD<c_4}
 &\leq 3c_2\(1\wedge\frac{\delta_D(x)\delta_{D}\(\frac{w+z}2\)}{ t}\)p\(t,x,\frac{w+z}2\),
 \end{align}
 for some $c_2(n,r,T)>0$. Finally, comparing \eqref{eq:pD>c_3} with \eqref{eq:pD<c_4} and taking $\delta(x)$ sufficiently small, we arrive at 
 $$\frac{\delta_D\(\frac{w+z}{2}\)}{\delta_{H_{w}}\(\frac{w+z}2\)}\geq \frac{c_1}{3c_2},$$
 valid for any $w,z\in\partial D$, which is equivalend to  $Q_D\geq \frac{c_1}{3c_2}>0$.
\end{proof}

After relaxing the condition $Q_D>0$ into $R_D>0$, the heat kernel $p_D(t,x,y)$ keeps admitting  two-sided estimates of the form of the upper bound from Theorem  \ref{thm:upper}.

\begin{thm}\label{thm:SF}
If    $D\in \mathcal S_R$, then  
\begin{align*}
p_D(t,x,y)
\stackrel{r,T,R_D}{\approx}  p(t,x,y)\[\(1\wedge \frac{\d(x)\d(y)}{t}\)+\(1\wedge \frac{\d_{H_x}(x)\d_{H_x}\(y\)}{t}\)\(1\wedge \frac{\d_{H_y}(y)\delta_{H_y}\(x\)}{t}\)\].
\end{align*}
holds for $x,y\in D$, $0<t<T$, where $H_x$, $H_y$ are any half-spaces such that $D\subset H_x,H_y$ and $\delta_D(x)=\delta_{H_x}(x)$, $\delta_D(y)=\delta_{H_y}(y)$.
\end{thm}
\begin{proof}
Due to Theorems \ref{thm:upper} and \ref{thm:lower} it is enough to show 
\begin{align}\label{eq:aux10}
p_D(t,x,y)\stackrel{r,T,R_D}{\gtrsim}\, p(t,x,y)\(1\wedge \frac{\d_{H_x}(x)\d_{H_x}\(y\)}{t}\)\(1\wedge \frac{\d_{H_y}(y)\delta_{H_y}\(x\)}{t}\).
\end{align}
Let $D\in \mathcal C^{1,1}_r(\R^n)$, $r>0$, be a strictly convex domain with $R_D>0$.  If $\delta_D(x), \delta_D(y)\geq\sqrt t$, the assertion follows from \eqref{eq:vdB}.

 Consider $\delta_D(x)\leq \sqrt t$ and $\delta_D(y)\geq 6\sqrt t$. Let $m$ be a point on the interval $\overline {x\(\frac{x+y}2\)}$ such that $\delta_D(m)=2\sqrt t$  and denote $\alpha=\frac{|x-m|}{|x-y|}\leq\frac13$. In order to take advantage of the assumption $D\in \mathcal S_R$, we need to choose suitably some points from the boundary. Indeed, let $x'\in \partial D$ be a point realising the distance of $x$ to $\partial D$ and let $y'\in\partial D$ be the other intersection point of $\partial D$ and the line containing $x'$ and $y$. Furthermore,  denote $m'=(1-\alpha)x+\alpha y$. Since $|m'-x'|\leq \alpha|x'-y|\leq\left|x'-\frac{x'+y}{2}\right|$ and  $\delta_D(m')\leq \delta_D(m)+\delta_D(x)\leq 3\sqrt t$  we get for $t\leq\frac19$
   $$|m'-x'|\leq \left|x'-\frac{x'+y'}{2}\right|\wedge \min\left\{|x'-z|:z\in \overline{x'y'}, \delta_D(z)\geq1\right\},$$
   and therefore  Proposition \ref{prop:mono}  gives us for $H_x=H_{x'}$
$$\delta_D(m')\geq R_D\delta_{H_x}(m')=\alpha R_D \delta_{H_x}(y).$$
 Then, for $z\in B(m,\sqrt t)$ we have 
$$\delta_D(z)\geq \frac12\delta_D(m)\geq  \frac13\delta_D(m')\geq\frac13\alpha R_D \delta_{H_x}(y).$$ 
Consequently, by Theorem \ref{thm:lower},  for such $z$ it holds
\begin{align*}
p_D\(\alpha t,x,z\)&\stackrel{r, R_D}{\gtrsim} \(1\wedge \frac{\d_{H_x}(x)\d_{H_x}\(y\)}{t}\)p\(\alpha t,z,y\),\ \ \ \ \ \ \ \ t<\frac19,
\end{align*}
and, by \eqref{eq:vdB},
\begin{align*}
p_D\(\(1-\alpha\)t,z,y\)\approx p\(\(1-\alpha\)t,z,y\).
\end{align*}
Thus, by virtue of Proposition  \ref{cor:CKlow}, we get
\begin{align}\notag
p_D(t,z,y)&=\int_{B(m,\sqrt t)}p_D\(\alpha t,x,z\)p_D\(\(1-\alpha\)t,z,y\)dz\\\notag
&\hspace{-5pt}\stackrel{r, R_D}{\gtrsim} \(1\wedge \frac{\d_{H_x}(x)\d_{H_x}\(y\)}{t}\)\int_{B(m,\sqrt t)}p\(\alpha t,x,z\)p\(\(1-\alpha\)t,z,y\)dz\\\label{eq:aux2}
&\gtrsim \(1\wedge \frac{\d_{H_x}(x)\d_{H_x}\(y\)}{t}\)p(t,x,y),
\end{align}
where $x,y\in D$ with $\delta_D(y)\geq 6\sqrt t$ and $t<\frac19$.

Consider now any $x,y\in D$.  For $T\leq\frac19$ small enough there exists a point $p$ such that $\left|p-\frac{x+y}2\right|=7\sqrt t$ and $\delta_D(p)\geq 7\sqrt t$. Then, for $z\in B(p,\sqrt t)$ we have $\delta_D(p)\geq 6\sqrt t$ and $$\d_{H_x}(z)\approx \d_{H_x}(p)\approx \d_{H_x}\(\frac{x+y}2\)\geq\frac12\d_{H_x}(y) ,\ \ \ \ \d_{H_y}(z)\approx \d_{H_y}(p)\approx \d_{H_y}\(\frac{x+y}2\)\geq\frac12\d_{H_y}(x) .$$  Hence, by Chapman-Kolmogorov identity, the estimate \eqref{eq:aux2} and Proposition \ref{cor:CKlow},
\begin{align*}
p_D(t,z,y)&=\int_{B(p,\sqrt t)}p_D\( t/2,x,z\)p_D\(t/2,z,y\)dz\\
&\stackrel{r,T, R_D}{\gtrsim} \(1\wedge \frac{\d_{H_x}(x)\d_{H_x}\(y\)}{t}\)\(1\wedge \frac{\d_{H_y}(y)\delta_{H_y}\(x\)}{t}\)p(t,x,y)
\end{align*}
as required. In order to extend the range of $T$ into any positive number, we proceed analogously as in the proof of Theorem \ref{thm:lower}.

\end{proof}

\subsection{Examples}

In this section we present some examples of sets/classes of sets and discuss the behaviour of their heat kernels.

\begin{prop}\label{prop:y>x^p}
Consider a domain   $U=\{x\in \R^n: x_n>a|(x_1,...,x_{n-1})|^p\}$, where $p\geq2$, $n\geq2 $ and $a>0$. Then $U\in \mathcal S_R$. 

As a consequence, the heat kernel $p_U(t,x,y)$ admits estimates from Theorem \ref{thm:SF} with constants depending on $n, T, a, p$. 
\end{prop}
\begin{proof}
 %In order to prove that $U\notin \mathcal S_Q$ it is enough to consider $w^{(k)}=(k,0,...,0,k^p)$ and $z^{(k)}=(k^2,0,...,0,k^{2p})$, $k\in \N$, which simplifies the problem to $2$-dimensional case. It is then an easy exercise to show that $\delta_U\(\frac{w^{(k)}+z^{(k)}}2\)\approx |z_1|\stackrel{a,p}{\approx} k^2$, while $\delta_{H_{w^{(k)}}}\(\frac{w^{(k)}+z^{(k)}}2\)\approx \frac{z_n^{(k)}}{ap\(w^{(k)}_1\)^{p-1}}\stackrel{a,p}{\approx}k^{p+1}$, which implies that 
%$$\lim_{k\rightarrow\infty} \frac{\delta_U\(\frac{w^{(k)}+z^{(k)}}2\)}{\delta_{H_{w^{(k)}}}\(\frac{w^{(k)}+z^{(k)}}2\)}=0,$$
%as required.

Since changing $a$ only rescales distances in the definition of $R_D$, we may assume  $a=1$. Consider $w,z\in \partial U$ and let us denote $\utilde w=(w_1,...,w_{n-1})$. Additionally, we assume, without loss of the generality, that $w=(w_1,0,...,0,w_n)$ with $w_{1},w_n\geq0$. The hyperplane $P_w$ tangent to $U$ at $w$ is then given by the equation
$$x_n=pw_{1}^{p-1}(x_1-w_1)+w_1^p.$$
First, we will show that  for $w,z$ satisfying   $|\utilde w-\utilde z|<\frac12 |\utilde w|$ it holds 
\begin{align}\label{eq:d/d>cp}
\frac{\delta_U\(\frac{w+z}{2}\)}{\delta_{H_{w}}\(\frac{w+z}2\)}>c_p,
\end{align}
for some constant $c_p>0$. Let $\delta_{H_w}^{\downarrow}(z)$, $z\in \partial U$, denote the distance between $z$ and its ''vertical'' projection onto $P_w$, i.e. projection along the vector $(0,...,0,-1)$,  given by
$$\delta_{H_w}^{\downarrow}(z)=z_n-\(pw_{1}^{p-1}(z_1-w_1)+w_1^p\).$$
Using the elementary  formula for the distance between a point and a  hyperplane, we get
$$\delta_{H_w}\(z\)=\frac{|z_n-\(pw_{1}^{p-1}(z_1-w_1)+w_1^p\)|}{\sqrt{\(pw_1^{p-1}\)^2+1}}=\frac{\delta^{\downarrow}_{H_{w}}\(z\)}{\sqrt{\(pw_1^{p-1}\)^2+1}},$$
which implies
$$\delta_{H_w}\(\frac{w+z}2\)=\frac12 \delta_{H_w}\(z\)\stackrel{p}{\approx}\frac{\delta^{\downarrow}_{H_{w}}\(z\)}{w_1^{p-1}+1}.$$
Furthermore, by two-dimensional Taylor's formula applied to $f(x,y)=((w_1+x)^2+y^2)^{p/2}$ we may write for $\utilde v=\utilde z-\utilde w$
\begin{align*}
\delta_{H_w}^{\downarrow}(z)&=\((w_1+v_1)^2+|\widetilde{\utilde v}|^2\)^{p/2}-\(pw_1^{p-1}v_1+w_1^p\)\\
&=v_1^2\frac p2\(((w_1+\xi_1)^2+\xi_2^2)^{\frac p2-1}+(p-2)(w_1+\xi_1)^2((w_1+\xi_1)^2+\xi_2^2)^{\frac p2-2}\)\\
&\ \ \ +|\widetilde{\utilde v}|^2\frac p2\(((w_1+\xi_1)^2+\xi_2^2)^{\frac p2-1}+(p-2)\xi_2^2((w_1+\xi_1)^2+\xi_2^2)^{\frac p2-2}\),
\end{align*}
where $|\xi_1|, |\xi_2|\leq |\utilde v|$, $\xi_1,\xi_2\in\R$. Hence, for $|\utilde v|/|\utilde w|<\frac12$  we have   $\delta_{H_w}^{\downarrow}(z)\stackrel{p}{\approx} w_1^{p-2}|\utilde v|^2$, and consequently
\begin{align*}
\delta_{H_w}\(\frac{w+z}{2}\)\stackrel{p}{\approx} \frac{w_1^{p-2}}{w_1^{p-1}+1}|\utilde w-\utilde z|^2.
\end{align*}
Next we will estimate the distance from $\frac{w+z}2$ to the boundary od $U$. One can see that if $\frac{w_n+z_n}2\leq 1$, then $\delta_U\(\frac{w+z}2\)\approx \delta_U^{\downarrow}\(\frac{w+z}2\)$, and if $\frac{w_n+z_n}2> 1$, then $\delta_U\(\frac{w+z}2\)\approx \delta_U^{\rightarrow}\(\frac{w+z}2\)$, where, analogously as previously, $\delta_{U}^{\downarrow}(u)$ denotes the distance between $u$ and its  projection onto $\partial U$ along the vector $(0,...,0,-1)$ and $\delta_{U}^{\rightarrow}(u)$ denotes the distance between $z$ and its  projection onto $\partial U$ along the vector $\utilde u$. In the case  $\frac{w_n+z_n}2\leq 1$ convexity of the function $\R\ni x\rightarrow x^{p/2}$ and the estimate $a^q-b^q {\approx}q(a-b)a^{q-1}$ valid for $0\leq b<a$ and $q\geq0$, give  us for $\utilde v=\utilde z-\utilde w$ and $|\utilde v|/|\utilde w|<\frac12$ 
\begin{align}\nonumber
\delta_U\(\frac{w+z}2\)&\approx \delta_U^{\downarrow}\(\frac{w+z}2\)=\frac12\(w_n+z_n\)-\left|\frac{\utilde w+\utilde z}2\right|^p\\\label{eq:aux15}
&=\(\frac12(w_1^2)^{p/2}+\frac12\big((w_1+v_1)^2+|\widetilde{\utilde v}|^2\big)^{p/2}\)-\((w_1+\frac12v_1)^2+|\frac12\widetilde{\utilde v}|^2\)^{p/2}\\\nonumber
&\geq\(\frac12w_1^2+\frac12(w_1+v_1)^2+\frac12|\widetilde{\utilde v}|^2\)^{p/2}-\((w_1+\frac12v_1)^2+|\frac12\widetilde{\utilde v}|^2\)^{p/2}\\\nonumber
&\stackrel{p}{\approx}\(\frac14v_1^2+\frac14|\widetilde{\utilde v}|^2\)\(\frac12w_1^2+\frac12(w_1+v_1)^2+\frac12|\widetilde{\utilde v}|^2\)^{\frac p2-1}\\[8pt]\nonumber
&\approx |\utilde v|^2w_1^{p-2}.
\end{align}
Furthermore, using the above bound, we get for $\frac{w_n+z_n}2> 1$
\begin{align*}
\delta_U\(\frac{w+z}2\)&\approx \delta_U^{\rightarrow}\(\frac{w+z}2\)=\(\frac{w_n+z_n}2\)^{1/p}-\left|\frac{\utilde w+\utilde z}2\right|\\
&\stackrel{p}{\approx}\({\frac12\(w_n+z_n\)-\left|\frac{\utilde w+\utilde z}2\right|^p}\)\(\frac{w_n+z_n}2\)^{\frac1p-1}\\
&\stackrel{p}{\approx }\delta_U^{\downarrow}\(\frac{w+z}2\)w_n^{\frac1p-1}\stackrel{p}{\gtrsim} |\utilde v|^2w_1^{-1}.
\end{align*}
This implies that for any value of $\frac{w_n+z_n}2\stackrel{p}{\approx}w_1^p $ we have 
$$\delta_U\(\frac{w+z}2\)\stackrel{p}{\gtrsim}\frac{w_1^{p-2}}{w_1^{p-1}+1}|\utilde w-\utilde z|^2\stackrel{p}{\approx}\delta_{H_w}\(\frac{w+z}{2}\),$$
which ends the proof of \eqref{eq:d/d>cp}.

Let us pass to the main part of the proof. Assume that $U\notin \mathcal S_R$, i.e. $R_U=0$. Then, there is a sequence of pairs $(w^{(k)},z^{(k)})_{k\geq1}$ from $\partial U\times\partial U$  such that $\lim_{k\rightarrow \infty}f_U(w^{(k)},z^{(k)})=0$, where 
$$f_U(w,z)=\left\{
\begin{array}{lll}
\frac{\delta_U\(\frac{w+z}{2}\)}{\delta_{H_{w}}\(\frac{w+z}2\)},&\text{ if }&\delta_{U}\(\frac{w+z}2\)\leq1,\\[10pt]
\sup_{\substack{m\in\overline{wz}\\\delta_U(m)>1}}\frac{\delta_U\(m\)}{\delta_{H_{w}}\(m\)},&\text{ if }&\delta_{U}\(\frac{w+z}2\)>1.
\end{array}\right.$$
Furthermore, there is a subsequence $(w^{(k_l)},z^{(k_l)})$ such that $|w^{(k_l)}|\rightarrow\infty$ or $w^{(k_l)}\rightarrow w^{(0)}$ for some $w^{(0)}\in \partial U$. Analogously, in each  case there  is a subsubsequence $(w^{(k_{l_m})},z^{(k_{l_m})})$ such that $|z^{(k_{l_m})}|\rightarrow\infty$ or $z^{(k_{l_m})}\rightarrow z^{(0)}$ for some $z^{(0)}\in \partial U$. We will show  that in any case $\liminf_{m\rightarrow\infty} f_U(w^{(k_{l_m})},z^{(k_{l_m})})>0$, which will finish the proof due to contradiction with the assumption $R_U=0$.

Assume first $w^{(k_l)}\rightarrow w^{(0)}$. If $z^{(k_{l_m})}\rightarrow z^{(0)}\neq w^{(0)}$, then $\liminf _{m\rightarrow \infty}f_U(w^{(k_{l_m})},z^{(k_{l_m})})\geq {\delta_U\(\frac{w^{(0)}+z^{(0)}}{2}\)}/{\delta_{H_{w^{(0)}}}\(\frac{w^{(0)}+z^{(0)}}2\)}>0$. In case $z^{(k_{l_m})}\rightarrow z^{(0)}=w^{(0)}$ there are  two options: for $|\utilde w -\utilde z|\leq \frac12|\utilde w|$  the inequality \eqref{eq:d/d>cp} ensures that $ f_U(w^{(k_{l_m})},z^{(k_{l_m})})>c_p$, while for $|\utilde w -\utilde z|> \frac12|\utilde w|$, which may happen only if $w^{(0)}=z^{(0)}=0$, we have  ${\delta_U\(\frac{w^{(k_{l_m})}+z^{(k_{l_m})}}{2}\)}\stackrel{p}{\approx }{\delta_{H_{w}}\(\frac{w^{(k_{l_m})}+z^{(k_{l_m})}}2\)}\stackrel{p}{\approx }\max\{|\utilde w|^p, |\utilde z|^p\}$. Eventually, for $|z^{(k_{l_m})}|\rightarrow\infty$ let us denote $m=(w^{(0)}_1,...,w^{(0)}_{n-1},m_n)$ such that $m_n>w^{(0)}_n$ and $\delta_U(m)=2$ and let $m^{(k_{l_m})}\in\overline{w^{(k_{l_m})}z^{(k_{l_m})}}$ be such that $m^{(k_{l_m})}_n=m_n$. We clearly have $m^{(k_{l_m})}\rightarrow m$ and therefore $\delta_U\(m^{(k_{l_m})}\)\rightarrow 2$, which implies $\liminf_{m\rightarrow\infty} f_U(w^{(k_{l_m})},z^{(k_{l_m})})\geq\frac2{\delta_{H_{w^{(0)}}}(m)}>0$.

Assume now $|w^{(k_l)}|\rightarrow \infty$. The bound \eqref{eq:d/d>cp} implies that if $|w^{(k_l)}_1-z^{(k_l)}_1|, |\widetilde{\utilde w}^{(k_l)}-\widetilde{\utilde z}^{(k_l)}|\leq \frac14|{\utilde w}^{(k_l)}|$  then $f_U(w^{(k_{l_m})},z^{(k_{l_m})})>c_p$. Furthermore, one can verify that for $ |\widetilde{\utilde w}^{(k_l)}-\widetilde{\utilde z}^{(k_l)}|> \frac14|{\utilde w}^{(k_l)}|$ we have ${\delta_U\(\frac{w^{(k_{l_m})}+z^{(k_{l_m})}}{2}\)}{\approx}{\delta_{H_{w}}\(\frac{w^{(k_{l_m})}+z^{(k_{l_m})}}2\)}\approx \max\{|w^{(k_{l_m})}_1|, |z^{(k_{l_m})}_1|\}$, so the remaining case is $|w^{(k_l)}_1-z^{(k_l)}_1|> \frac14|{\utilde w}^{(k_l)}|= \frac14|{ w_1}^{(k_l)}|$. Choosing $m^{(k_{l_m})}\in\overline{w^{(k_{l_m})}z^{(k_{l_m})}}$ such that $|m^{(k_{l_m})}_n-w^{(k_{l_m})}_n|=\frac18 w^{(k_{l_m})}_n$ we have  ${\delta_U\(m^{(k_{l_m})}\)}{\approx}{\delta_{H_{w}}\(m^{(k_{l_m})}\)}\approx w^{(k_{l_m})}_1\rightarrow\infty$, which may be observed, for instance, considering $z=\(\frac54w_1,0,...,0,  (\frac54w_1)^p \)$ and $z=\(\frac34w_1,0,...,0,  (\frac34w_1)^p \)$. This ends the proof.

\end{proof}

\begin{prop}\label{prop:analytical}
For $n=2$ the class $\mathcal S_Q$ contains strictly convex bounded domains with analytical boundary. As a consequence, the heat kernels of such  sets admit estimates from both: Theorem \ref{thm:SF}  and Theorem \ref{thm:SQ}.
\end{prop}
\begin{proof}
The idea of the proof is similar as in the case of previous proposition. Let $D$ be a strictly convex bounded domain with analytical boundary and suppose there is a sequence $(w^{(k)},z^{(k)})_{k\geq1}$, $w^{(k)},z^{(k)}\in \partial D$, such that 
\begin{align}\label{eq:star}
\liminf_{k\rightarrow \infty}\frac{\delta_D\(\frac{w^{(k)}+z^{(k)}}{2}\)}{\delta_{H_{w^{(k)}}}\(\frac{w^{(k)}+z^{(k)}}2\)}=0.
\end{align}
 Compactness of $\partial D$ allows us assume that  $\lim_{k\rightarrow \infty}w^{(k)}=w^{(0)}$ and $\lim_{k\rightarrow \infty}z^{(k)}=z^{(0)}$ for some $w^{(0)}, z^{(0)}\in \partial D$. Since \eqref{eq:star} is clearly not satisfied for $w^{(0)}\neq z^{(0)}$, we therefore assume also $w^{(0)}= z^{(0)}$. Next, we  rotate and translate $D$ such  that $w^{(0)}=(0,0)$ and $D\subset \R\times(0,\infty)$. Since $\partial D$ is analytical,  there exists a neighbourhood $G$ of $(0,0)$ such that $\partial D\cap G$ is a graph of a function $y=f(x)=\sum_{i=i_0}^\infty a_ix^i$ where $i_0\geq 2$, $a_{i_0}>0$, $(x,y)\in G$. In particular, this implies $\lim _{x\rightarrow 0}\frac{f(x)}{a_{i_0}x^{i_0}}=1$.
Hence, for $|w_1^{(k)}-z_1^{(k)}|>\frac12|w_1^{(k)}|$ we have ${\delta_D\(\frac{w^{(k)}+z^{(k)}}{2}\)}{\approx}{\delta_{H_{w}}\(\frac{w^{(k)}+z^{(k)}}2\)}\approx \max\{|w^{(k)}_1|^{i_0}, |z^{(k)}_1|^{i_0}\}$, $k\geq k_0$ for some $k_0\in\N$. Furthermore, denoting $u=w_1^{(k)}$ and $v=z^{(k)}_1-w_1^{(k)}$, Taylor's formula gives us  for $|w_1^{(k)}-z_1^{(k)}|\leq\frac12|w_1^{(k)}|$
\begin{align}\nonumber
\delta_{H_{w^{(k)}}}\(\frac{w^{(k)}+z^{(k)}}2\)&=\frac12\delta_{H_{w^{(k)}}}\(z^{(k)}\)\\\nonumber
&=\cos\(\arctan\(f'(u+v)\)\)\Big(f(u+v)-\(f(u)+vf'(u)\)\Big)\\\nonumber
&\approx \sum_{i=i_0}^\infty a_i\((u+v)^i-\big(u^i+iu^{i-1}v\big)\)\\\label{eq:aux3}
&= \sum_{i=i_0}^\infty a_i\frac{i(i-1)}2\(u+\xi_i\)^{i-2}v^2\approx |u|^{i_0-2}v^2,
\end{align}
for some $|\xi_i|\leq \frac12|u|$. Similarly, 
\begin{align*}
\delta_D\(\frac{w^{(k)}+z^{(k)}}2\)&\approx \frac12\(f(w^{(k)}_1)+f(z_1^{(k)})\)-f\(\frac{w_1^{(k)}+z_1^{(k)}}2\)\\
&=\sum_{i=i_0}^\infty a_i\(\frac12u^i+\frac12\(u+v\)^i-\(u+\frac12v\)^i\)\\
&\geq a_{i_0}\(\frac12u^{i_0}+\frac12\(u+v\)^{i_0}-\(u+\frac12v\)^{i_0}\)\\
&\ \ \ \ \ -\sum_{i=i_0+1}^\infty |a_i|\(\frac12u^i+\frac12\(u+v\)^i-\(u+\frac12v\)^i\).
\end{align*}
Taking $|\widetilde{\utilde v}|=0$ in \eqref{eq:aux15}, we get
$$a_{i_0}\(\frac12u^{i_0}+\frac12\(u+v\)^{i_0}-\(u+\frac12v\)^{i_0}\)\stackrel{i_0}{\gtrsim}|u|^{i_0-2}v^2. $$
Furthermore, Taylor's formula used twice implies 
\begin{align*}
&\left|\sum_{i=i_0+1}^\infty |a_i|\(\frac12u^i+\frac12\(u+v\)^i-\(u+\frac12v\)^i\)\right|\\
&\stackrel{i_0}{\lesssim}\sum_{i=i_0+1}^\infty |a_i|i(i-1)\(|u|+|v|\)^{i-2}v^2\approx |u|^{i_0-1}v^2,
\end{align*}
and we conclude $\delta_D\(\frac{w^{(k)}+z^{(k)}}2\)\gtrsim |w^{(k)}_1|^{i_0-2}v^2$ for $|w^{(k)}|$ small enough.  Thus, in view of \eqref{eq:aux3}, for such $w^{(k)}$ it holds
$\delta_{H_{w^{(k)}}}\(\frac{w^{(k)}+z^{(k)}}2\)/\delta_D\(\frac{w^{(k)}+z^{(k)}}2\)\geq c(i_0)>0$, which  contradicts  \eqref{eq:star}.

\end{proof}

Let us define the following set
$$ S=\big(B_2((-1,0),1)\big)\cup\big((-1,1)\times(-1,1)\big)\cup \big(B_2((1,0),1)\big) \subset \R^2,$$
which is a square $(-1,1)\times(-1,1)$ with two semicircles added to its left and right sides. It is known as a stadium. The next example shows that for some range of arguments the heat kernel $p_S(t,x,y)$ is comparable neither to the bound from Theorem \ref{thm:upper} nor to the one from Theorem \ref{cor:mainlower}. Note that the space arguments realising the  indicated behaviour of $p_S(t,x,y)$   are located at opposite ends of the 'flat' part of the boundary, which strongly suggests that non-strict convexity is indeed the property that impacts on the incomparability of the  bounds.
\begin{exmp}\label{exmpl:1} Let $x,y\in S$ such that $x_1<-1$, $y_1>1$,  $x_2,y_2=1-t^{\gamma}$, $\gamma>0$, and $\delta_S(x),\delta_S(y)<t^{1+\gamma}$ with $t<1$. For $0<\gamma\leq\frac12$ we have
\begin{align*}
p_S(t,x,y)\approx p(t,x,y)\[\(1\wedge \frac{\d(x)\d(y)}{t}\)+\(1\wedge \frac{\d_{H_x}(x)\d_{H_x}\(y\)}{t}\)\(1\wedge \frac{\d_{H_y}(y)\delta_{H_y}\(x\)}{t}\)\],
\end{align*}
and  for $\gamma\geq\frac23$ it holds
$$p_S(t,x,y)\approx p(t,x,y)\[\(1\wedge\frac{\delta_S(x)\delta_S(y)}{t}\)+\(1\wedge \frac{\d_{S}(x)\d_{S}\(\frac{x+y}2\)}{t}\)\(1\wedge \frac{\d_{S}(y)\delta_{S}\(\frac{x+y}{2}\)}{t}\)\].$$
However, for $\frac12<\gamma<\frac23$ we have
\begin{align}\label{eq:aux12}
p_S(t,x,y)\approx \frac{\delta_S(x)\delta_S(y)}{t^{3(1-\gamma)}}p(t,x,y),
\end{align}
while 
$$\(1\wedge\frac{\delta_S(x)\delta_S(y)}{t}\)+\(1\wedge \frac{\d_{S}(x)\d_{S}\(\frac{x+y}2\)}{t}\)\(1\wedge \frac{\d_{S}(y)\delta_{S}\(\frac{x+y}{2}\)}{t}\)\approx \frac{\delta_S(x)\delta_S(y)}{t}$$
and 
$$\(1\wedge \frac{\d(x)\d(y)}{t}\)+\(1\wedge \frac{\d_{H_x}(x)\d_{H_x}\(y\)}{t}\)\(1\wedge \frac{\d_{H_y}(y)\delta_{H_y}\(x\)}{t}\)\approx \frac{\delta_S(x)\delta_S(y)}{t^{2-\gamma}}= \frac{\delta_S(x)\delta_S(y)}{t^{3(1-\gamma)+2\(\gamma-\tfrac12\)}}.$$
\end{exmp}
\begin{proof}
First, let us observe 
 \begin{align}\label{eq:aux4}
 \d_{S}\(\frac{x+y}2\)=t^{\gamma},\ \ \ \ \ \ \ \text{ and }\ \ \ \ \  \ \   \ \d_{H_x}(y)\approx \d_{H_y}(x)\approx t^{\gamma/2},
 \end{align}
 which immediately follows the  last two approximations in the assertion.
  
 Assume $\gamma\leq \frac23$ and put $\alpha=\frac{|x-(-1,x_2)|}{|x-y|}$. We have $\alpha \approx t^{\gamma/2}$ and therefore $\sqrt{\alpha t}<c t^{(2+\gamma)/4}\leq ct^\gamma$, $t<t_0$, for some $c, t_0>0$. Chapman-Kolmogorov identity, Theorem \ref{thm:lower} and Proposition \ref{cor:CKlow} give us for some $\varepsilon>0$
\begin{align*}
p_S(t,x,y)&\geq \int_{B_2\((-1,x_2),\frac12t^\gamma\)}\int_{B_2\((1,x_y),\frac12t^\gamma\)}p_S(\alpha t,x,w)p_S((1-2\alpha)t,w,z)p_S(\alpha t,z,y)dzdw\\[8pt]
&\gtrsim \(1\wedge \frac{\delta_S(x)t^{\gamma}}{\alpha t}\)\(1\wedge\frac{t^{2\gamma}}{t}\)\(1\wedge \frac{\delta_S(y)t^{\gamma}}{\alpha t}\)\\
&\ \ \ \int_{B_2\((-1,x_2),\frac1{2c}\sqrt {\alpha t}\)}\int_{B_2\((1,x_y),\frac1{2c}\sqrt {\alpha t}\)}p(\alpha t,x,w)p((1-2\alpha)t,w,z)p(\alpha t,z,y)dzdw\\[8pt]
&\stackrel{c}{\approx}  \(1\wedge \frac{\delta_S(x)\delta_{H_x}(y)}{ t}\)\(1\wedge{t^{2\gamma-1}}\)\(1\wedge \frac{\delta_S(y)\delta_{H_y}(x)}{ t}\)p(t,x,y)\\
&\approx  \frac{\delta_S(x)\delta_S(y)}{ t^{2-\gamma}}\(1\wedge{t^{2\gamma-1}}\)p(t,x,y).
\end{align*}
This implies the lower bound in \eqref{eq:aux12} and, together with Theorem \ref{thm:upper} and \eqref{eq:aux4}, the first estimate in the example. Furthermore, the lower bound in the latter estimate follows directly from \eqref{eq:lower2}. In order to obtain remaining upper bounds for $\gamma \geq \frac12$, we denote $H=\{x\in\R^2:x_2<1\}$ and  estimate $\delta_{H_x}(z)\delta_{H_y}(z), \lesssim \delta_H(z)+ t^{\gamma/2}$ for $z\in S$. Consequently, by Theorem \ref{thm:upper}, 
\begin{align*}
&p_S(t/2,x,z)\\
&\lesssim \(\(1\wedge \frac{\delta_S(x)\delta_S(z)}{t}\)+\(1\wedge \frac{\delta_S(x)\(\delta_H(z)+t^{\gamma/2}\)}{t}\)\(1\wedge\frac{\delta_H(z)t^{\gamma}}{t}\)\)p(t/2,x,z)\\
&\lesssim \frac{\delta_S(x)\delta_H(z)}{t}\(1+t^{\gamma-1}\(\delta_H(z)+t^{\gamma/2}\)\)p(t/2,x,z).
\end{align*}
Similarly, $p_S(t/2,z,y)\lesssim \frac{\delta_S(x)\delta_H(z)}{t}\(1+t^{\gamma-1}\(\delta_H(z)+t^{\gamma/2}\)\)p(t/2,z,y)$, hence
\begin{align*}
I_1&\lesssim  \leq   \frac{\delta_S(x)\delta_S(y)}{t^{2}}\int_{H}\(\delta_{H }(z)\)^2\(1+t^{2\gamma-2}\(\delta^2_{H }(z)+t^{\gamma}\)\)p(t/2,x,z)p(t/2,z,y)dz\\
&\lesssim \frac{\delta_S(x)\delta_S(y)}{t}\(1+t^{2\gamma-1}+t^{3\gamma-2}\)p(t,x,y),
\end{align*}
where the last inequality follows from  Proposition \ref{prop:CKHH} with $H_1=H_2=H$. 
For $\gamma\geq \frac23$, we get $p_S(t,x,t)\lesssim \frac{\delta_S(x)\delta_S(y)}{t}p(t,x,y)$, which, combined with \eqref{eq:lower2}, implies the second estimate in the example. Finally, for $\frac12<\gamma<\frac23$,  we have $\(1+t^{2\gamma-1}+t^{3\gamma-2}\)\approx t^{3\gamma-2}$,
which follows the upper bound in \eqref{eq:aux12}. The proof is complete.
 \end{proof}

\section*{Acknowledgements} The author would like to thank Jacek Ma{\l}ecki for inspiration and some fruitful discussion.\\
This work was supported by by the National Science Centre, Poland, grant no. 2015/18/E/ST1/00239

\end{document}